\documentclass{amsart}
\usepackage{amsmath}
\usepackage{amsthm}
\usepackage{listings}
\usepackage{paralist}
\usepackage{bm}
\usepackage{xcolor,soul}
\usepackage{tikz-cd}
\usetikzlibrary{decorations.markings}
\tikzset{negated/.style={
        decoration={markings,
            mark= at position 0.5 with {
                \node[transform shape] (tempnode) {$\backslash$};
            }
        },
        postaction={decorate}
    }
}
\usepackage{mathtools}
\usepackage{extpfeil}

\newtheorem{theorem}{Theorem}[section]
\newtheorem{proposition}[theorem]{Proposition}
\newtheorem{lemma}[theorem]{Lemma}
\usepackage{amssymb}

\usepackage{algorithm2e}
\usepackage{enumitem}  
\usepackage{stmaryrd}

\setenumerate{
label={\normalfont(\textit{\roman*})},
ref={(\textrm{\roman*)}},
wide
}

\makeatletter
\newcommand{\myitem}[1]{\item[#1]\protected@edef\@currentlabel{#1}}

\usepackage{ifthen}
\usepackage{scalerel}
\usepackage{hyperref}
\usepackage{mathrsfs}
\usepackage[mathscr]{euscript}
\usepackage{lineno}

\newtheorem{corollary}[theorem]{Corollary}

\theoremstyle{definition}

\newtheorem*{remark*}{Remark}

\newtheorem{conjecture}[theorem]{Conjecture}

\renewcommand{\Re}{\operatorname{Re}}

\newcommand{\lcm}{\operatorname{lcm}}

\newcommand{\linspan}{\operatorname{span}}

\newcommand{\fp}[1]{\left\{ #1 \right\} }
\newcommand{\ip}[1]{\left\lfloor #1 \right\rfloor }

\newcommand{\bra}[1]{\left(#1\right)}
\newcommand{\brabig}[1]{\big(#1\big)}

\newcommand{\braBig}[1]{\Big( #1 \Big)}
\renewcommand{\tilde}{\widetilde}
\renewcommand{\bar}{\overline}

\newcommand{\abs}[1]{\left|#1\right|}
\newcommand{\absnormal}[1]{|#1|}
\newcommand{\absbig}[1]{\big|#1\big|}

\newcommand{\set}[2]{\left\{ #1 \ \middle| \ #2 \right\} }

\newcommand{\ceil}[1]{\left\lceil #1 \right\rceil}

\newcommand{\e}{\varepsilon}
\renewcommand{\a}{\alpha}
\renewcommand{\b}{\beta}

\newcommand{\NN}{\mathbb{N}}

\newcommand{\QQ}{\mathbb{Q}}
\newcommand{\PP}{\mathbb{P}}

\newcommand{\FF}{\mathbb{F}}

\newcommand{\ZZ}{\mathbb{Z}}
\newcommand{\RR}{\mathbb{R}}
\newcommand{\CC}{\mathbb{C}}

\newcommand{\cG}{\mathcal{G}}
\newcommand{\cH}{\mathcal{H}}
\newcommand{\cD}{\mathcal{D}}

\newcommand{\cF}{\mathcal{F}}
\newcommand{\cS}{\mathcal{S}}

\newcommand{\cL}{\mathcal{L}}
\newcommand{\cN}{\mathcal{N}}

\newcommand{\cZ}{\mathcal{Z}}

\newcommand{\cM}{\mathcal{M}}

\newcommand{\dlog}{{d}_{\log}}

\definecolor{fresh}{HTML}{269e00}
\definecolor{checked}{HTML}{1e5e06}
\definecolor{double}{HTML}{5E3800}
\definecolor{external}{HTML}{a81a78}
\definecolor{later}{HTML}{0410ff}
\definecolor{minor-rev}{HTML}{d96a09}
\definecolor{major-rev}{HTML}{c90000}
\definecolor{skip}{HTML}{ffffff}\definecolor{normal}{HTML}{000000}

\definecolor{fresh}{HTML}{000000}
\definecolor{checked}{HTML}{000000}
\definecolor{double}{HTML}{000000}
\definecolor{external}{HTML}{000000}
\definecolor{later}{HTML}{000000}
\definecolor{minor-rev}{HTML}{000000}
\definecolor{major-rev}{HTML}{000000}

\newcommand{\bb}{\mathbf}

\renewcommand{\subset}{\subseteq}

\newcommand*\patchAmsMathEnvironmentForLineno[1]{\expandafter\let\csname old#1\expandafter\endcsname\csname #1\endcsname
  \expandafter\let\csname oldend#1\expandafter\endcsname\csname end#1\endcsname
  \renewenvironment{#1}{\linenomath\csname old#1\endcsname}{\csname oldend#1\endcsname\endlinenomath}}\newcommand*\patchBothAmsMathEnvironmentsForLineno[1]{\patchAmsMathEnvironmentForLineno{#1}\patchAmsMathEnvironmentForLineno{#1*}}\AtBeginDocument{\patchBothAmsMathEnvironmentsForLineno{equation}\patchBothAmsMathEnvironmentsForLineno{align}\patchBothAmsMathEnvironmentsForLineno{flalign}\patchBothAmsMathEnvironmentsForLineno{alignat}\patchBothAmsMathEnvironmentsForLineno{gather}\patchBothAmsMathEnvironmentsForLineno{multline}}

\newcounter{claimcounter}
\counterwithin*{claimcounter}{theorem}
\newenvironment{claim}{\refstepcounter{claimcounter}{\vspace{3pt}\par\noindent\textit{Claim \theclaimcounter:}}}{\vspace{3pt}}

\newenvironment{claimproof}[1]{\par\noindent\textit{Proof:}\space#1}{\hfill $\triangle$ \vspace{3pt}}

\newcommand{\D}{\mathbb{D}}

\begin{document}

\author[J.\ Konieczny ]{Jakub Konieczny}
\address{Universit\'e Claude Bernard Lyon 1, CNRS UMR 5208, Institut Camille Jordan,
F-69622 Villeurbanne Cedex, France}
\email{jakub.konieczny@gmail.com}

\title{On asymptotically automatic sequences}
\begin{abstract}
	We study the notion of an asymptotically automatic sequence, which generalises the notion of an automatic sequence. While $k$-automatic sequences are characterised by finiteness of $k$-kernels, the $k$-kernels of asymptotically $k$-automatic sequences are only required to be finite up to equality almost everywhere. We prove basic closure properties and a linear bound on asymptotic subword complexity, show that results concerning frequencies of symbols are no longer true for the asymptotic analogue, and discuss some classification problems.
\end{abstract}


\maketitle 

\section{Introduction}\label{sec:Intro}\color{checked}

{\color{double}
Automatic sequences are a widely-studied class of sequences, which can be defined in many equivalent ways, each shedding light from a different perspective. Perhaps the shortest definition uses the notion of the $k$-kernel. For a sequence $\bb a = (a_n)_{n=0}^\infty$ and a base $k \geq 2$, the \emph{$k$-kernel} of $\bb a$ is the set
\[
	\cN_k(\bb a) := \set{ \bra{a_{k^i n+r}}_{n=0}^\infty }{ i,r \in \NN_0,\ r < k^i }.
\]
A sequence $\bb a$ is $k$-automatic if and only if its $k$-kernel is finite, $\#\cN_k(\bb a) < \infty$. 

We will say that two sequences $\bb a = (a_n)_{n=0}^\infty$ and $\bb b = (b_n)_{n=0}^\infty$ are \emph{asymptotically equal}, denoted $\bb a \simeq \bb b$, if $a_n = b_n$ for almost all $n \in \NN_0$, i.e.,
\[
	\frac{1}{N} \#\set{ 0 \leq n < N}{ a_n \neq b_n} \to 0 \text{ as } N \to \infty.
\]
In \cite{JK-Cobham-asympt}, the author introduced the notion of an \emph{asymptotically $k$-automatic sequence}, which is a sequence $\bb a$ over a finite alphabet whose $k$-kernel is finite up to asymptotic equality. In other words, we require that the set
\[
	\tilde \cN_k(\bb a) := \cN_k(\bb a)/{\simeq}
\]
should be finite, $\# \tilde \cN_k(\bb a) < \infty$, or equivalently that there exists a finite number of sequences $\bb a^{(0)}, \bb a^{(1)}, \dots, \bb a^{(d-1)}$ such that for each $\bb b \in \cN_k(\bb a)$ there exists $0 \leq j < d$ such that $\bb b \simeq \bb a^{(j)}$. Asymptotically $k$-automatic sequences include $k$-automatic sequences, their modifications on density zero set of positions, as well as sequences $\bb a = (a_n)_{n=0}^\infty$ of the form
\(
	a_n = F\bra{\fp{\log_k n}},
\) 
where $F \colon [0,1) \to \Omega$ is a Riemann-measurable map taking values in some finite set $\Omega$. More examples are also constructed in \cite{JK-Cobham-asympt}.

The purpose of this paper is to systematically investigate asymptotically automatic sequences, and to establish analogues of various well-known facts about automatic sequences --- or to disprove them. There are several sources of motivation behind this endeavour. Firstly, we believe that the class of asymptotically automatic sequences is interesting in its own right and worthy of investigation. Secondly, some of the new results obtained here immediately imply new density versions of results expressed in terms of automatic sequences. In particular, this is the case for classification results obtained in Section \ref{sec:class}. Finally, we hope that the arguments presented here will better illuminate the theory of automatic sequences.
}

{\color{double}
In Section \ref{sec:basic} we discuss basic properties of asymptotically $k$-automatic sequences. In particular, we prove closure under Cartesian products and codings, and under passing to an arithmetic progression. We also discuss the connection with automata, obtained in \cite{JK-Cobham-asympt}.

In Section \ref{sec:bases} we discuss the dependence of the notion of a $k$-automatic sequence on the base $k$. We show that for multiplicatively dependent bases $k$ and $\ell$, asymptotically $k$-automatic sequences are the same as asymptotically $\ell$-automatic sequences. For multiplicatively independent bases $k$ and $\ell$, a celebrated theorem of Cobham asserts that the only sequences that are automatic in both bases are the eventually periodic ones. In contrast, as shown in \cite{JK-Cobham-asympt}, there exist sequences that are asymptotically automatic in two multiplicatively independent bases without being asymptomatically equal to a periodic sequence; we can only show a considerably weaker property of asymptotic invariance under a shift. We investigate the set of bases with respect to which a given sequence is asymptotically automatic and show that it has a particularly structured form.

In Section \ref{sec:freq} we discuss frequencies of symbols and subwords in asymptotically automatic sequences. In the case of automatic sequences, frequencies are not guaranteed to exist, but logarithmic frequencies are. Additionally, if frequencies do exist then they are rational. In contrast, we construct examples which show that for asymptotically automatic sequences logarithmic frequencies also are not guaranteed to exist, and if frequencies exist they are not guaranteed to be rational.

In Section \ref{sec:subword} we discuss subword complexity and its newly introduced asymptotic analogue. Automatic sequences are notable for having linear subword complexity. We argue that subword complexity is not the correct notion to consider in the asymptotic regime, and show that asymptotically automatic sequences have linear asymptotic subword complexity. We also prove several basic facts about asymptotic subword complexity. In particular, we classify sequences whose asymptotic subword complexity is especially low, which is analogous to the classification of periodic and Sturmian sequences in terms of subword complexity.

In Section \ref{sec:class} we discuss classification problems. Specifically, we fully classify bracket words that are asymptotically automatic, in particular showing that Sturmian words are never asymptotically automatic. We also partially classify multiplicative sequences that are asymptotically automatic.

\color{later}
In Section \ref{sec:regular} we introduce the notion of an asymptotically $k$-regular sequence, and we prove a variant of Cobham's theorem.
}

\subsection*{Acknowledgements}\color{later}
The author is grateful to Boris Adamczewski, Jakub Byszewski, Daniel Krenn, Oleksiy Klurman, Clemens M\"ullner for helpful comments and discussions. The author works within the framework of the LABEX MILYON (ANR-10-LABX-0070) of Universit\'e de Lyon, within the program "Investissements d'Avenir" (ANR-11-IDEX-0007) operated by the French National Research Agency (ANR).

\subsection{Notation}\color{double}

We let $\NN = \{1,2,\dots\}$ denote the set of natural numbers and let $\NN_0 = \NN \cup \{0\}$. For $N \in \NN_0$, we let $[N] = \{0,1,\dots,N-1\}$ denote the length-$N$ initial interval of $\NN_0$. 

For a set $\Omega$ we let $\Omega^*$ denote the set of all finite words $w = w_{0}w_{1}\dots w_{\ell-1}$ ($\ell \in \NN_0$, $w_0,w_1,\dots,w_{\ell-1} \in \Omega$) over $\Omega$, including the empty word $\epsilon$. The length of a word $w$ is denoted by $\abs{w}$. We let $\Omega^\infty$ denote the set of all sequences $\bb w = w_0w_1 w_2 \dots$ over $\Omega$. Given a sequence $\bb w$, we let $w_i$ denote its $i$-th entry (starting with $0$) and for $i \leq j$ we let $w_{[i,j)} = w_i w_{i+1} \dots w_{j-1}$ denote the corresponding subword.

For a word $w \in \Omega^*$ and $n \in \NN_0$, we let $w^n = ww\dots w$ denote the $n$-fold repetition of $w$, and $w^\infty = ww\dots \in \Omega^\infty$ denote the periodic sequence with period $w$. More generally, for $t \in \RR_{+}$ we let $w^t$ denote the prefix of $w^\infty$ of length $\ip{ t \abs{w}}$. (We caution that, in general, $(w^{t})^s \neq w^{ts}$.)

For an integer $k \geq 2$, we let $\Sigma_k = \{0,1,\dots,k-1\}$ denote the set of base-$k$ digits. For an integer $n \in \NN_0$, we let $(n)_k$ denote its base-$k$ expansion (without leading zeros) and for $u \in \Sigma_k^*$ (possibly with leading zeros) we let $[u]_k$ denote the integer encoded by $[u]$. For $i \in \NN_0$ we also let $(n)_k^i$ denote the unique word $u \in \Sigma_k^i$ such that $n \equiv [u]_k \bmod{k^i}$ (this is the length-$i$ suffix of $(n)_k$ or the result of padding $(n)_k$ with leading zeros, depending on the length of $(n)_k$).

We use standard asymptotic notation. For two quantities $X,Y$, we write $X = O(Y)$ or $X \ll Y$ if $\abs{X} \leq C Y$ for a constant $C$. If $C$ is additionally allowed to depend on a parameter $A$, we write $X = O_A(Y)$ or $X \ll_A Y$. Assuming that $X$ and $Y$ depend on a parameter $N$, we write $X = o_{N \to \infty}(Y)$ if $X/Y \to 0$ and $X = \omega_{N \to \infty}(Y)$ if $X/Y \to +\infty$ as $N \to \infty$; if $N$ is clear from the context and there is no risk of confusion, we simply write $o(Y)$ and $\omega(Y)$ instead.

We say that a statement $\varphi(n)$ holds for almost all $n \in \NN$ if the set of $n \in \NN$ for which $\varphi(n)$ is false has density zero, i.e., $\#\set{n < N}{ \neg \varphi(n)} = o(N)$ as $N \to \infty$.
\color{black} 
\section{Basic properties}\label{sec:basic}\color{double}

In this section, we review some basic properties of asymptotically automatic sequences. Most of them are direct analogues of similar properties of automatic sequences and can be established using essentially the same argument.

\begin{lemma}\label{lem:basic:product}
	Let $k \geq 2$ and let $\bb a,\bb b$ be asymptotically $k$-automatic sequences. Then the Cartesian product $\bb a \times \bb b = \bra{(a_n,b_n)}_{n=0}^\infty$ is an asymptotically $k$-automatic sequence.
\end{lemma}
\begin{proof}
	It is enough to notice that $\cN_k(\bb a \times \bb b) \subset \cN_k(\bb a) \times \cN_k(\bb b)$.
\end{proof}	

\begin{lemma}\label{lem:basic:coding}
	Let $k \geq 2$, let $\bb a$ be an asymptotically $k$-automatic sequence over a finite alphabet $\Omega$ and let $\rho \colon \Omega \to \Omega'$ be a map. Then the coding $\rho(\bb a) = \bra{\rho(a_n)}_{n=0}^\infty$ is an asymptotically $k$-automatic sequence.
\end{lemma}
\begin{proof}
	It is enough to notice that $\cN_k(\rho(\bb a))\subset \set{\rho(\bb b)}{\bb b \in \cN_k(\bb a)}$.
\end{proof}	

\begin{corollary}\label{lem:basic:sgroup}
	Let $k \geq 2$ and let $(S,\cdot)$ be a semigroup. Then asymptotically $k$-automatic sequences over $S$ form a semigroup, with the operation defined pointwise, i.e.\ by $\bb a \cdot \bb b = (a_n \cdot b_n)_{n=0}^\infty$.
\end{corollary}
\begin{proof}
	Apply Lemmas \ref{lem:basic:product} and \ref{lem:basic:coding} with $\rho \colon S \times S \to S$ given by $\rho(x,y) = x \cdot y$.
\end{proof}

In particular, the family of complex-valued asymptotically $k$-automatic sequences is closed under addition and multiplication, i.e.\ it constitutes a ring. Similarly, the multiplicative inverse of an asymptotically $k$-automatic sequence, if it exists, is asymptotically $k$-automatic.

\begin{lemma}\label{lem:basic:almost-eq}
	Let $k \geq 2$, let $\bb a$ be an asymptotically $k$-automatic sequence and let $\bb a'$ be a sequence with $\bb a' \simeq \bb a$. Then $\bb a'$ is asymptotically $k$-automatic.
\end{lemma}
\begin{proof}
	It is enough to notice that $\tilde \cN_k(\bb a) = \tilde \cN_k(\bb a')$.
\end{proof}

	As a consequence of Lemma \ref{lem:basic:almost-eq}, it makes sense to speak of asymptotic $k$-automaticity of a sequence that is only defined almost everywhere. For instance, in the following result we consider the restriction $\bb a' = \bra{ a_{qn+r} }_{n=0}^\infty$ of an asymptotically $k$-automatic sequence $\bb a$ to an arithmetic progression $(qn+r)_{n=0}^\infty$. When $r < 0$, a few initial terms of $\bb a'$ are not well-defined. However, we can assign arbitrary values to those entries, and whether or not the resulting sequence is asymptotically $k$-automatic does not depend on the choices made.

\begin{lemma}\label{lem:basic:arith-prog}
	Let $k \geq 2$, let $\bb a$ be a sequence over a finite alphabet, and let $q \in \NN$. For $r \in \ZZ$, let   $\bb a^{(r)} := \bra{ a_{qn+r} }_{n=0}^\infty$, where we let $a_n := a_0$ for $n < 0$. 
	\begin{enumerate}
	\item\label{it:basic:AP-1} If $\bb a$ is asymptotically $k$-automatic then so is $\bb a^{(r)}$ for each $r \in \ZZ$.
	\item\label{it:basic:AP-2} If for all $r$ with $0 \leq r < q$ the sequence $\bb a^{(r)}$ is asymptotically $k$-automatic then $\bb a$ is also asymptotically $k$-automatic .
	\end{enumerate}
\end{lemma}
\begin{proof}
	\begin{enumerate}
	\item It will suffice to consider the special cases where $q = 1$ and $r = \pm 1$, and where $r = 0$. The general case can be recovered by repeated applications of these two cases. 
In the case where $r = 0$, the proof follows from the same argument as Theorem 6.8.1 in \cite{AlloucheShallit-book}, except that we replace $\cN_k$ with $\tilde \cN_k$. The key idea is that the set 
	\[
		\set{ \bra{b_{qn+s}}_{n=0}^\infty}{\bb b \in \cN_k(\bb a),\ 0 \leq s < q}
	\]
	is finite up to equality almost everywhere and contains $ \cN_k\bra{\bb a^{(0)}}$.
In the case where $q = 1$ and $r = \pm 1$, similarly, the proof follows from the same argument as Theorem 6.8.3 in \cite{AlloucheShallit-book}. The key idea is that the set
	\[
	 	\cN_k(\bb a) \cup \set{ \bb b^{(r)}}{\bb b \in \cN_k(\bb a)},
	\]
	is finite up to equality almost everywhere and contains $\cN_k(\bb a^{(r)})$. (Here, $\bb b^{(r)}$ is defined by $b^{(r)}_n := b_{qn+r} = b_{n \pm 1}$ if $qn+r \geq 0$ and $b^{(r)}_n := b_0$ otherwise.)
	
	\item Pick any $\bb b \in \cN_k(\bb a)$, say $b_{n} = a_{k^i n + m}$ with $m,i \in \NN_0$, $m < k^i$. Then $\bb b$ is uniquely determined by the sequences $\bb b^{(s)}$ for $s$ with $0 \leq s < q$. Pick any $0 \leq s < q$ and put $r := k^i s + m \bmod q$ and $h := \ip{\bra{ k^i s + m}/q}$. Then $h < k^i$ and 
	\[
		b^{(s)}_n = a_{qk^i n + k^i s + m} = a^{(r)}_{k^i n + h}
	\]
	Thus, $\bb b^{(s)} \in \cN_k(\bb a^{(r)})$. Since $\tilde \cN_k(\bb a^{(r)})$ is finite and $r$ is chosen from a finite set, we conclude that $\tilde \cN_k(\bb a)$ is finite. \qedhere
	\end{enumerate}
\end{proof}

For a sequence $\bb a$ taking values in an commutative monoid we may consider the sequence of partial sums $\Sigma \bb a$ given by $(\Sigma a)_n = \sum_{m=0}^{n-1} a_m$ (in particular, $a_0 = 0$). The following lemma is an analogue of Corollary 6.9.3 in \cite{AlloucheShallit-book}.
\begin{lemma}\label{lem:basic:partial-sums}
	Let $\bb a$ be an asymptotically $k$-automatic sequence over a finite abelian monoid $(S,+)$. Then the sequence $\Sigma \bb a$ is asymptotically $k$-automatic.
\end{lemma}
\begin{proof}
	For any sequence $\bb b$ and $0 \leq i < k$, let $\bb b^{(i)}$ be the sequence given by $b^{(i)} := b_{kn+i}$. For each $0 \leq j \leq k$ we have 
	\begin{align*}\label{eq:basic:24:1}
		(\Sigma b)_{kn+j} 
		&= \sum_{m = 0}^{n-1} \sum_{i = 0}^{k-1} b_{km+i} + \sum_{i=0}^{j-1} b_{kn+i} 
		\\& = \sum_{i = 0}^{k-1} (\Sigma b^{(i)})_{n} + \sum_{i=0}^{j-1} b^{(i)}_n 
		= \sum_{i = 0}^{j-1} (\Sigma b^{(i)})_{n+1} + \sum_{i = j}^{k-1} (\Sigma b^{(i)})_{n}.
	\end{align*}
	Let $\cS \subset S^\infty$ be the sub-monoid of $S^\infty$ generated by $\Sigma \bb b$ and their shifts by $1$ (i.e., $( (\Sigma b)_{n+1})_{n=0}^\infty$) for $\bb b \in \cN_k(\bb a)$, with the operation defined pointwise. Since $S$ and $\tilde\cN_k(\bb a)$ are finite, $\cS/{\simeq}$ is finite. Applying the identity obtained above, we see that $\cS$ is closed under the operation $\bb c \mapsto \bra{c_{kn+i}}_{n=0}^\infty$ for each $0 \leq i < k$. It follows that $ \cN_k(\Sigma \bb a) \subset \cS $. Thus, $\tilde \cN_k(\Sigma \bb a)$ is finite.
\end{proof}

We close this section by citing a result from \cite{JK-Cobham-asympt} which elucidates the connection between automata and asymptotically automatic sequences. For a map $\phi \colon \Sigma_k^* \to \Omega$, we define the \emph{$k$-kernel} (or simply the \emph{kernel}, since there is no risk of ambiguity)
\begin{align*}
	\cN_k(\phi) &= \set{ \phi_v }{ v \in \Sigma_k^*}, & 
	\phi_v(u) := \phi(uv) \text{ for } u,v \in \Sigma_k^*.
\end{align*}
The map $\phi$ is automatic if and only if its kernel is finite, $\#\cN_k(\phi) < \infty$.

\begin{proposition}[{\cite{JK-Cobham-asympt}}]\label{prop:basic:structure}
	Let $\bb a$ be an asymptotically $k$-automatic sequence. Then there exists $d \in \NN$ and an automatic map $\phi \colon \Sigma_k^* \to \Sigma_d$ as well as sequences $\bb a^{(0)},\bb a^{(1)}, \dots, \bb a^{(d-1)}$ such that for each $u \in \Sigma_k^*$ we have
	\[
		\brabig{a_{k^{\abs{u}}n + [u]_k}}_{n=0}^\infty \simeq \brabig{a^{(\phi(n))}_n}_{n=0}^\infty.
	\]
\end{proposition}

In applications, we will always assume that the value of $d$ is minimal, meaning that no pair of sequences $\bb a^{(0)},\bb a^{(1)}, \dots, \bb a^{(d-1)}$ is asymptotically equal. The following can be construed as an analogue of the pumping lemma for automatic sequences. 

\begin{lemma}
	Let $\bb a$ be an asymptotically $k$-automatic sequence. Then there exist $p \in \NN$ such that for each word $w \in \Sigma_k^*$ with $\abs{w} \geq p$ there exist words $u,v,v' \in \Sigma_k^*$ such that $w = vuv'$, $\abs{w}-p \leq \abs{u} < \abs{w}$ and, letting $w(t) := v u^t v'$, for all $t \in \NN_0$ we have
$	
	\bra{a_{k^{\abs{w}} n + [w]_k}}_{n=0}^\infty \simeq \bra{a_{k^{\abs{w(t)}} n + [w(t)]_k}}_{n=0}^\infty$.
\end{lemma} 
\begin{proof}
	Apply Proposition \ref{prop:basic:structure} and then apply the pumping lemma \cite[Lem.\ 4.2.1]{AlloucheShallit-book} to $\phi$.
\end{proof}
 
\section{Bases}\label{sec:bases}\color{double}

A classical theorem due to Cobham asserts that for an automatic sequence there is, essentially, only one base with respect to which it is automatic. To be more precise, it is not hard to show that each eventually periodic sequence is automatic in every base and that for any pair of multiplicatively dependent integers $k,\ell \geq 2$ (i.e., $\log k/\log \ell \in \QQ$), $k$-automatic sequences are the same as $\ell$-automatic sequences. Cobham's theorem asserts that for each automatic but not eventually periodic sequence $\bb a$ there exists precisely one integer $k \geq 2$ that is not a perfect power and such that $\bb a$ is $k$-automatic.

In contrast, an example from \cite{JK-Cobham-asympt} shows that there exists a sequence $\bb a$ that is asymptotically $2$- and $3$-automatic, but not asymptotically periodic. (It also seems very plausible that this sequence is not $k$-automatic for $k$ that have at least one prime factor $p \geq 5$, but the proof of this statement remains elusive.) Thus, for a given sequence over a finite alphabet, it makes sense to inquire into the set of bases with respect to which it is asymptotically automatic.

We first note that, in analogy with automatic sequences, asymptotically automatic sequences in multiplicatively dependent bases are the same. The argument is virtually identical to the one for automatic sequences, but we include it here for the sake of completeness.

\begin{lemma}\label{lem:bases:dependent}
	Let $\bb a$ be a sequence over a finite alphabet $\Omega$ and let $k,\ell \geq 2$ be multiplicatively dependent integers. Then $\bb a$ is asymptotically $k$-automatic if and only if $\bb a$ is asymptotically $\ell$-automatic.
\end{lemma}
\begin{proof}
	It is enough to consider the case where $\ell = k^m$ is an integer power of $k$ with $m \geq 2$. In this case we have $\cN_\ell(\bb a) \subset \cN_k(\bb a)$ and hence also $\tilde\cN_\ell(\bb a) \subset \tilde\cN_k(\bb a)$. Thus, if $\bb a$ is $k$-automatic then it is also $\ell$-automatic. Conversely, suppose that $\bb a$ is $\ell$-automatic, and consider a sequence $\bb b \in \tilde\cN_{k}(\bb a)$, say $b_{n} = a_{k^i n + r}$ where $0 \leq r < k^i$. Put $j = i \bmod m$ and $s = \ip{r/\ell^{\ip{i/m}}}$. Bearing in mind the identity
	\[
		k^{i} n + r = \ell^{\ip{i/m}}(k^j n + s) + \bra{r \bmod \ell^{\ip{i/m}}},
	\]
	we see that $\bb b \in \tilde\cN_{\ell}\bra{\bb a'}$, where $a'_{n} = a_{k^j n + s}$. By Lemma \ref{lem:basic:arith-prog}, $\bb a'$ is asymptotically $\ell$-automatic, and hence $ \tilde\cN_{\ell}\bra{\bb a'}$ is finite. Since $\bb b \in \tilde\cN_k(\bb a)$ was arbitrary and $j$ and $s$ are uniformly bounded with respect to $\bb b$, it follows that $\tilde\cN_k(\bb a)$ is finite.
\end{proof}

Next, we show some closure properties of the set of bases with respect to which a given sequence is asymptotically automatic. 	

\begin{lemma}\label{lem:bases:semigrp}
Let $k,\ell \geq 2$ be multiplicatively independent integers and let $\bb a$ be an asymptotically $k$-automatic sequence. Then $\bb a$ is asymptotically $\ell$-automatic if and only if it is asymptotically $k\ell$-automatic.
\end{lemma}
\begin{proof}
	Suppose first that $\bb a$ is asymptotically $\ell$-automatic. To show that $\bb a$ is asymptotically $k\ell$-automatic it is enough to notice that 
		\[ 
			\cN_{k \ell}(\bb a) \subset \bigcup_{\bb b \in \cN_\ell(\bb a)} \cN_{k}(\bb b),
		\]
		and each sequence $\bb b \in \cN_\ell(\bb a)$ is asymptotically $k$-automatic by Lemma \ref{lem:basic:arith-prog}.

Suppose next that $\bb a$ is asymptotically $k\ell$-automatic. Let $I \in \NN_0$ be sufficiently large that for each integer $R$ with $0 \leq R < k^I$ there exist infinitely many pairs $(i,r) \in \NN_0^2$ with $0 \leq r < k^i$ such that $\bra{a_{k^i n+r}}_{n=0}^\infty \simeq \bra{a_{k^I n+R}}_{n=0}^\infty$. We will show that for each $0 \leq R < k^I$ the sequence $\bra{a_{k^I n+R}}_{n=0}^\infty$ is asymptotically $\ell$-automatic, which will finish the argument by Lemma \ref{lem:basic:arith-prog}. Each sequence in $\cN_{\ell}\bra{\bra{a_{k^I n+R}}_{n=0}^\infty}$ takes the form $\bra{a_{k^I \ell^j n+t}}_{n=0}^\infty$ for some $j,t \in \NN_0$ with $0 \leq t < k^I \ell^j$ and $t \equiv R \bmod k^I$. By the definition of $I$, we can find arbitrarily large $i \in \NN_0$ and $0 \leq r < k^i$ such that
\begin{align*}
\bra{a_{k^I \ell^j n+t}}_{n=0}^\infty =
\bra{a_{k^I (\ell^j n + \ip{t/k^I})+R}}_{n=0}^\infty \simeq
\bra{a_{k^i (\ell^j n + \ip{t/k^I})+r}}_{n=0}^\infty =
\bra{a_{k^i \ell^j n + u}}_{n=0}^\infty,
\end{align*}
where $u  = k^i\ip{t/k^I}+r < k^i \ell^j$. Since $\bb a$ is asymptotically $k\ell$-automatic, we can find $J,T \in \NN_0$, with $J$ bounded by a constant independent of $j$ and $0 \leq T < (k\ell)^J$ such that 
\begin{align*}
\bra{a_{k^i \ell^j n + u}}_{n=0}^\infty
& =
\bra{a_{ (k\ell)^j(k^{i-j} n + \ip{u/(k\ell)^j}) + (u \bmod (k\ell)^j) }}_{n=0}^\infty 
\\ & \simeq
\bra{a_{ (k\ell)^J(k^{i-j} n + \ip{u/(k\ell)^j}) + T }}_{n=0}^\infty
= \bra{a_{ k^{i-j+J} \ell^J n + v }}_{n=0}^\infty,
\end{align*}
where $v = (k\ell)^J \ip{u/(k\ell)^j}+T < k^{i-j+J} \ell^J$. Finally, using the fact that $\bb a$ is asymptotically $k$-automatic, we can find $M,V \in \NN_0$ with $M$ bounded by a constant independent of $j$ and $0 \leq V < k^{M}$ such that 
\begin{align*}
\bra{a_{ k^{i-j+J} \ell^J n + v }}_{n=0}^\infty 
&= 
\bra{a_{ k^{i-j+J} (\ell^J n + \ip{v/k^{i-j+J}}) + (v \bmod k^{i-j+J}) }}_{n=0}^\infty 
\\ &\simeq 
\bra{a_{ k^{M} (\ell^J n + \ip{v/k^{i-j+J}}) + V }}_{n=0}^\infty 
=
\bra{a_{ k^{M} \ell^J n + S}}_{n=0}^\infty,
\end{align*}
where $S =  k^{M} \ip{v/k^{i-j+J}} + V < k^{M} \ell^J$. Ultimately, the sequence $\bra{a_{k^I \ell^j n+t}}_{n=0}^\infty$ is asymptotically equal to the sequence $\bra{a_{ k^{M} \ell^J n + S}}_{n=0}^\infty$, which is taken from a finite set independent of $j$ and $t$. This shows  $\tilde\cN_{\ell}\bra{\bra{a_{k^I n+R}}_{n=0}^\infty}$ is finite, as needed.
\end{proof}

As a consequence of Lemma \ref{lem:bases:semigrp}, the set of bases with respect to which $\bb a$ is asymptotically automatic has a particularly structured form. Before we proceed, we need the following simple lemma on vector spaces.

\begin{lemma}\label{lem:bases:vectors}
	Let $S \subset \bigoplus_{i=1}^\infty \QQ$ be a set satisfying the following properties:
\begin{enumerate}
\item for each $u \in S$ and $i \in \NN$, $u_i \in \NN_0$;
\item for each $u \in S$ and $m \in \NN$, if $m \mid u_i$ for all $i \in \NN$ then $(1/m)u \in S$;
\item for each $u,v \in S$, $u+v \in S$;
\item for each $u,v \in S$, if $u_i \geq v_i$ for all $i \in \NN$ then $u-v \in S$.
\end{enumerate}
Then there exists a vector space $V < \bigoplus_{i=1}^\infty \QQ$ such that 
\[
	S = V \cap \NN_0^\infty = \set{ v \in V}{ v_i \in \NN_0 \text{ for all } i \in \NN}.
\]
\end{lemma}
\begin{proof}
	Let $V$ denote the set of all $v \in \bigoplus_{i=1}^\infty \QQ$ such that there exist $m \in \NN$ and $u \in S$ such that $m v+u \in S$. It is routine to check that if $v,v' \in V$ and $n \in \NN$ then $v + v' \in V$, $(1/n) v \in V$ and $-v \in V$. Thus, $V$ is a vector space.
	Directly by definition, for all $v \in S$ we have $v \in V$ and $v_i \in \NN_0$ for all $i \in \NN$. Suppose conversely that $v \in V$ and $v_i \in \NN_0$ for all $i \in \NN_0$. Then $m v + u \in S$ for some $u \in S$ and $m \in \NN$. We then infer from the closure properties of $S$ that $mv \in S$ and $v \in S$.  
\end{proof}

For a prime $p$, we let $\nu_p$ denote the $p$-adic valuation, meaning that $p^{\nu_p(n)}$ is the largest power of $p$ that divides $n$. Let also $p_1, p_2,\dots$ be the increasing enumeration of the primes.

\begin{corollary}\label{cor:bases:form}
	Let $\bb a$ be a sequence over a finite alphabet $\Omega$.
	There exists a vector subspace $V < \bigoplus_{i=1}^\infty \QQ$ such that for $k \in \NN$ with $k > 1$, $\bb a$ is asymptotically $k$-automatic if and only if $(\nu_{p_1}(k), \nu_{p_2}(k), \dots) \in V$.
\end{corollary}
\begin{proof}
	By Lemma \ref{lem:bases:semigrp}, the set
	\[
		S = \{(0,0,\dots) \} \cup \set{ (\nu_{p_1}(k), \nu_{p_2}(k), \dots) }{ \bb a \text{ is asymptotically $k$-automatic}}
	\] 
	satisfies the assumptions of Lemma \ref{lem:bases:vectors}.
\end{proof}

It seems plausible that the following converse to Corollary \ref{cor:bases:form} is also true. 

\begin{conjecture}
	Let $V < \bigoplus_{i=1}^\infty \QQ$ be a vector subspace. Then there exists a sequence $\bb a$ over the alphabet $\{0,1\}$ such that for $k \in \NN$ with $k > 1$, $\bb a$ is asymptotically $k$-automatic if and only if $(\nu_{p_1}(k), \nu_{p_2}(k), \dots) \in V$.
\end{conjecture}

Let us briefly explain the rationale behind the conjecture above. Pick a set of primes $P$. We will construct a candidate for a sequence that is asymptotically $k$-automatic if and only is $k$ is a product of primes from $P$. (The fact that we are dealing specifically with primes does not seem to be crucial, but it makes the discussion somewhat simpler.)

Assume first that $P = \{p_1,p_2,\dots,p_r\}$ is finite. Let $\cD \subset \NN$ be the set of all products of primes in $P$,
\[
	\cD = \set{ p_1^{i_1} p_2^{i_2} \dots p_r^{i_r} }{ i_1,i_2,\dots,i_r \in \NN_0}.
\]
We use $\cD$ to partition $\NN$ into intervals $I = [N,M)$ where $N < M$ are consecutive elements of $\cD$. Let us say that the interval $I$ mentioned above is \emph{exceptional} if there exists a prime $p \in P$ such that $p \nmid N$ or $p \nmid M$. If $I$ is not exceptional then for any $p \in P$, the integers $N/p$  $M/p$ are again consecutive elements of $\cD$. This motivates us to consider a sequence $\bb a$ that is constant between consecutive elements of $\cD$. For concreteness, writing $N = p_1^{i_1} p_2^{i_2} \dots p^{i_r}$ and $M = p_1^{j_1} p_2^{j_2} \dots p^{j_r}$, we can set
\begin{align*}
	a_n &= i_1 + i_2 + \dots + i_r + j_1 + j_2 + \dots + j_r \bmod 1 \text{ for } n \in I.
\end{align*}
If $I$ is not exceptional then for $n \in I$ we have $a_{n} = a_{\ip{n/p}}$. In other words, letting $E$ denote the union of all exceptional intervals, we have $a_{n} = a_{\ip{n/p}}$ for all $n \in \NN \setminus E$. We conjecture that $E$ has density zero; this is trivial when $r = 1$, a simple lemma proved in \cite{JK-Cobham-asympt} when $r = 2$, and open for $r \geq 3$. If this conjecture is true then the sequence $\bb a$ defined above is asymptotically $p$-automatic for all $p \in P$, and hence also for all $k$ that are products of primes in $P$. (Indeed, $\tilde \cN_p(\bb a)$ has only one element.) On the other hand, when $k$ has prime factors not belonging to $P$, there seems to be no particular reason why $\bb a$ should be asymptotically $k$-automatic --- hence, we can reasonably conjecture that  that $\bb a$ is not asymptotically $k$-automatic. 

In the case where $P = \{p_1,p_2,\dots\}$ is infinite, we modify the construction above, setting
\[
	\cD = \set{ N = p_1^{i_1} p_2^{i_2} \dots p_r^{i_r} }{ i_1,i_2,\dots,i_r \in \NN_0, r \leq \ell(N)},
\]
where $\ell \colon \NN  \to \NN$ is non-decreasing and $\ell(N)$ tends to infinity very slowly as $N \to \infty$. We then proceed like above. Assuming that the discussion in the finite case goes through (meaning that the two conjectures we put forward are indeed true) then we can derive the infinite case by a limiting argument. (Since the discussion is largely conjectural, we skip further details.)

\newcommand{\freqsup}{\overline{\operatorname{freq}}}
\newcommand{\freqinf}{\underline{\operatorname{freq}}}
\newcommand{\freq}{{\operatorname{freq}}}

\newcommand{\freqsuplog}{\overline{\operatorname{freq}}^{\log}}
\newcommand{\freqinflog}{\underline{\operatorname{freq}}^{\log}}
\newcommand{\freqlog}{{\operatorname{freq}}^{\log}}

\newcommand{\good}{\mathrm{g}}
\newcommand{\bad}{\mathrm{b}}

\section{Frequencies of symbols}\label{sec:freq}

\subsection{Logarithmic frequencies}\color{double}

For an sequence $\bb a$ over a finite alphabet $\Omega$ and for a symbol $\omega \in \Omega$, we define the lower and upper (asymptotic) frequencies of $\omega$ in $\bb a$:
\begin{align}\label{eq:freq:def-freq}
\freqsup_{\bb a}(\omega) &:= \overline{d}\bra{\set{n \in \NN_0}{a_n = \omega }} = \limsup_{N \to \infty} \frac{\# \set{n \in [N]}{a_n = \omega} }{N}, \\
\freqinf_{\bb a}(\omega) &:= \underline{d}\bra{\set{n \in \NN_0}{a_n = \omega }} = \liminf_{N \to \infty} \frac{\# \set{n \in [N]}{a_n = \omega} }{N}.
\end{align} 
More generally, for a word $w \in \Omega^\ell$, we define
\begin{align}\label{eq:freq:def-freq-word}
\freqsup_{\bb a}(w) &:= \overline{d}\bra{\set{n \in \NN_0}{ a_{[n,n+\ell)} = w }}, \\
\freqinf_{\bb a}(w) &:= \underline{d}\bra{\set{n \in \NN_0}{a_{[n,n+\ell)} = w }}.
\end{align} 
When $\freqsup_{\bb a}(w) = \freqinf_{\bb a}(w)$, the common value is called the frequency of $w$ is $\bb a$ and is denoted by $\freq_{\bb a}(w)$.
Replacing $\overline{d}$ and $\underline{d}$ with upper and lower logarithmic densities given by
\begin{align*}
	\overline{d}^{\log}(A) &= \limsup_{N \to \infty} \frac{1}{\log N} \sum_{n=0}^{N-1} \frac{1_A(n)}{n+1}, &
	\underline{d}^{\log}(A) &= \liminf_{N \to \infty} \frac{1}{\log N} \sum_{n=0}^{N-1} \frac{1_A(n)}{n+1},	
\end{align*}
we similarly obtain the notions of upper and lower logarithmic frequencies
\begin{align}\label{eq:freq:def-freq-word-log}
\freqsuplog_{\bb a}(w) &:= \overline{d}^{\log}\bra{\set{n \in \NN_0}{ a_{[n,n+\ell)} = w }}, \\
\freqinflog_{\bb a}(w) &:= \underline{d}^{\log}\bra{\set{n \in \NN_0}{a_{[n,n+\ell)} = w }},
\end{align} 
as well as the notion of logarithmic frequency, $\freqlog_{\bb a}(w) =\freqsuplog_{\bb a}(w) = \freqinflog_{\bb a}(w)$. 

When $\bb a$ is automatic, the frequencies of symbols are not guaranteed to exist, as shown by the basic example $a_n = \abs{(n)_2} \bmod 2$. However, the logarithmic frequencies exist for all symbols, and more generally for all words \cite[Thm.\ 8.4.9]{AlloucheShallit-book}.
Thus, it may come as a surprise that there exist asymptotically automatic sequences such that logarithmic frequencies of symbols do not exist.

\begin{proposition}\label{prop:freq:not-converge}
	There exists an asymptotically $2$-automatic sequence $\bb a$ over the alphabet $\{0,1\}$ such that $\freqsuplog_{\bb a}(1) = 1$ and $\freqinflog_{\bb a}(1) = 0$. In particular,  $\freqlog_{\bb a}(1)$ does not exist.
\end{proposition}
\begin{proof}

	Following the construction in \cite[Sec.\ 4]{JK-Cobham-asympt}, let
\(
(H_i)_{i=0}^\infty = \bra{2^{\a_i} 3^{\b_i} }_{i=0}^\infty 
\)
 be the increasing enumeration of the set 
\begin{equation}\label{eq:70:cH}
\cH = \set{2^\a 3^\b}{\a,\b \in \NN_0} = \{ 1,2,3,4,6,8,9,12,\dots\}.
\end{equation}
Let $(\gamma_j)_{j=0}^\infty$ be an increasing sequence of integers with $\gamma_0 = 0$, which will be specified in the course of the argument. 
We define the sequence $\bb a$ by putting $a_0 = 0$ and for each $j \in \NN_0$ and each $i \in \NN_0$ such that $\b_i \in [\gamma_j,\gamma_{j+1})$ letting
\begin{equation}\label{eq:70:def-f}
	a_n = j \bmod 2 \in \{0,1\} \text{ for } n \in [H_i,H_{i+1}).
\end{equation}
It follows by the same argument as in  \cite[Lem.\ 4.2]{JK-Cobham-asympt} that 
\begin{equation}\label{eq:70:inv-f}
	a_n = a_{n+1} = a_{2n} \text{ for almost all } n \in \NN_0
\end{equation}
and as a consequence $\bb a$ is asymptotically $2$-automatic. 

We plan to ensure that for each $j \in \NN$, letting $N = 3^{\gamma_j}$, we have
\begin{equation}\label{eq:70:limit-j}
	\frac{1}{\log N}\sum_{n=0}^{N-1} \frac{a_n}{n+1} \in 
	\begin{cases}
		[0,2^{-j}) & \text{if }  j \equiv 1 \bmod 2,\\
		(1-2^{-j},1] & \text{if } j \equiv 0 \bmod 2.
	\end{cases} 	
\end{equation}
Once \eqref{eq:70:limit-j} has been proved, letting $j \to \infty$ we conclude that  $\freqsuplog_{\bb a}(1) = 1$ and $\freqinflog_{\bb a}(1) = 0$, as needed. We also note that \eqref{eq:70:limit-j} only depends on $\gamma_0,\gamma_1,\dots,\gamma_j$ and not on $\gamma_{j+1},\gamma_{j+2},\dots$ (this is because each $n \in [1,N) = [1,3^{\gamma_{j}})$ belongs to an interval $[H_i,H_{i+1}) \subset [1,3^{\gamma_{j}})$ which satisfies $\beta_i < \gamma_j$). Our plan is to let $\gamma_j$ grow very rapidly with $j$. Thus, it will suffice to show that, given any values of $\gamma_0,\gamma_1,\dots,\gamma_{j-1}$, condition \eqref{eq:70:limit-j} holds for all sufficiently large values of $\gamma_j$. For concreteness, suppose that $j \equiv 1 \bmod 2$. For each interval $[H_i,H_{i+1}) \subset [1,3^{\gamma_j})$, if $\beta_i \geq \gamma_{j-1}$ then we have $a_n = 0$ for all $n \in [H_i,H_{i+1})$. Using the rough estimate $a_n \leq 1$ for all remaining $n$ and applying Lemma \ref{lem:freq:beta-is-large} below concludes the argument.
\end{proof}

\begin{lemma}\label{lem:freq:beta-is-large}
	For each $\beta \in \NN_0$, we have
	\begin{equation}\label{eq:935:1}
		\lim_{k \to \infty} \frac{1}{\log H_k} 
		\sum_{ i < k,\, \beta_i = \beta } \sum_{n = H_i}^{H_{i+1}-1} \frac{1}{n+1} = 0.
	\end{equation}
\end{lemma}
\begin{proof}
	Let $S_k$ denote the sum in \eqref{eq:935:1} and pick $\e > 0$. By \cite[Lem.\ 4.1]{JK-Cobham-asympt} we have $H_{i+1}/H_i \to 0$ as $i \to \infty$. Hence, we can pick $i_1$ such that $H_{i+1} \leq (1+\e)H_i$ for all $i \geq i_1$. Splitting the outer sum in \eqref{eq:935:1} at $i_1$ and estimating $1/(n+1)$ from above by $1/H_i$ for $n \in [H_i,H_{i+1})$, we obtain
	\begin{equation}\label{eq:935:2}
		S_k \leq O_{\e}(1) + \e \cdot \# \set{ i }{ i_1 \leq i < k,\ \beta_i = \beta} \leq O_{\e}(1) + \e \log_2(H_k). 
	\end{equation}
	It follows that 
	\begin{equation}\label{eq:935:3}
		\limsup_{k \to \infty} \frac{S_k}{\log H_k} \leq \frac{\e}{\log 2}.
	\end{equation}
	Letting $\e \to 0$ concludes the proof.
\end{proof}

\subsection{Rationality of frequencies}

Another notable property of automatic sequences is that frequencies of symbols, if they exist, are rational \cite[Thm.\ 8.4.5(b)]{AlloucheShallit-book}. We show that this property also does not extend to asymptotically automatic sequences, which can have arbitrary frequencies of symbols. In particular, this shows that there are uncountably many distinct asymptotically automatic sequences (up to equality almost everywhere and renaming symbols), in contrast with automatic sequences which are easily seen to be countable.

\begin{proposition}
	For each $\theta \in [0,1]$, there exists an asymptotically $2$-automatic sequence $\bb a$ over the alphabet $\{0,1\}$ such that $\freq_{\bb a}(1) = \theta$.
\end{proposition}
\begin{proof}
	For each $n \in \NN_0$, there is a unique decomposition of the binary expansion
	\begin{equation}\label{eq:freq:58:1}
		(n)_2 = u^{(1)}_{n} u^{(2)}_n \dots u^{(r(n))}_{n} v_{n},
	\end{equation}
	where $r(n) \in \NN_0$, $u_{n}^{(i)} \in \Sigma_2^* 1$ and $\absbig{u_{n}^{(i)}}_1 = i$ for $i = 1,2,\dots,r(n)$, and $v_n \in \Sigma_2^*$ and $\absbig{v_n}_1 \leq r(n)$. We will take $\bb a$ of the form
	\begin{equation}\label{eq:freq:58:2}
		a_n = f( u_n^{(r(n))}),
	\end{equation}	
	where $f \colon \Sigma_2^* \to \{0,1\}$ remains to be specified.

\begin{claim}\label{claim:freq:A}
	We have $\# \tilde \cN_2(\bb a) = 1$. In particular, $\bb a$ is asymptotically $2$-automatic.
\end{claim}
\begin{claimproof}
	We need to prove that for almost all $n \in \NN_0$ we have
	\begin{equation}\label{eq:freq:58:4}
		a_n = a_{2n} = a_{2n+1}.
	\end{equation}
	The decomposition of $(2n)_2$ as in \eqref{eq:freq:58:1} takes the form
	\begin{equation}\label{eq:freq:58:3}
		(n)_2 = u^{(1)}_{n} u^{(2)}_n \dots u^{(r(n))}_{n} (v_{n}0),
	\end{equation}
	and thus $u^{(r(2n))}_{2n} = u^{(r(n))}_{n}$ and $a_n = a_{2n}$ for all $n$.
	
	To prove the second equality in \eqref{eq:freq:58:4}, we begin by noting that, similarly to \eqref{eq:freq:58:3}, the decomposition of $(2n+1)_2$ as in \eqref{eq:freq:58:1} takes the form
	\begin{equation}\label{eq:freq:58:5}
		(n)_2 = u^{(1)}_{n} u^{(2)}_n \dots u^{(r(n))}_{n} (v_{n}1),
	\end{equation}
	unless $\absbig{v_{n}}_1 = r(n)$, in which case 
	\begin{equation}\label{eq:freq:58:6}
		s_2(n) = \absbig{(n)_2}_1 = 1 + 2 + \dots + r(n) + r(n) = \bra{r(n)^2+3r(n)}/2.
	\end{equation}
	Let $R = \set{\bra{r^2+3r}/2}{r \in \NN_0}$. It will suffice to show that $s_2(n) \not \in R$ for almost all $n \in \NN_0$.
	 
	 Let $L$ be a large integer and consider a random variable $\bb n$ uniformly distributed on $[2^L]$. Then $s_2(\bb n)$ has binomial distribution $\operatorname{Bin}(L,1/2)$. By Hoeffding's inequality, we have 
	\begin{equation}\label{eq:freq:58:7}
	 \PP\bra{ \abs{s_2(\bb n) - L/2} > \sqrt{L\log L} } \leq 2/L^{2}.
	\end{equation}
	A standard estimate on binomial coefficients, which is easily derived from Stirling's formula, asserts that
	\begin{equation}\label{eq:freq:58:8}
	\max_{x \in \ZZ} \PP\brabig{s_2(\bb n) = x} = 2^{-L}\binom{L}{\ip{L/2}} \ll \frac{1}{\sqrt{L}} .
	\end{equation}	 
	Thus, bearing in mind that $\# R \cap [N] = \sqrt{2N} + O(1)$, we can estimate
	\begin{align*}
	\PP\brabig{ s_2(\bb n) \in R} &\leq 
	\sum_{ \abs{x - L/2} \leq \sqrt{L\log L} } 1_R(x) \PP\brabig{s_2(\bb n) = x} 
	+ \PP\bra{ \abs{s_2(\bb n) - L/2} > \sqrt{L\log L} }
	\\ & \ll \sqrt{\frac{\log L}{L}} + \frac{1}{L^2} \to 0 \text{ as } L \to \infty.
	\end{align*}
	It follows that $s_2(n) \not \in R$ for almost all $n \in \NN_0$, as needed.
\end{claimproof}	
	
We let $f$ be any map from $\Sigma_2^*$ to $\{0,1\}$ with the property that for each $r,m \in \NN_0$ with $r \leq m$ we have
\begin{align}\label{eq:freq:def-of-f}
	\# \set{ w \in \Sigma_2^m}{ f(w) = 1,\ \abs{w}_1 = r } = \ip{\theta \binom{m}{r}}.
\end{align}
	Such a map can easily be constructed by picking, for each $r, m \in \NN_0$ with $r \leq m$, a partition $\set{ w \in \Sigma_2^m}{ \abs{w}_1 = r } = X_0 \cup X_1$ into two sets with cardinalities $\# X_1 = \ip{\theta \binom{m}{r}}$ and $\# X_0 = \ceil{(1-\theta) \binom{m}{r}}$, and then putting $f(w) = 1$ for $w \in X_1$ and $f(w) = 0$ for $w \in X_0$.
	
\begin{claim}\label{claim:freq:B}
	We have $\freq_{\bb a}(1) = \theta$.
\end{claim} 
\begin{claimproof}	
	Let $N$ be a large integer, and let $L = \ceil{\log_2 N}$. We need to approximate $\# \set{ n \in [N]}{ a_n = 1}$.	
	Towards this end, we partition $[N]$ into cells 
	\begin{align}\label{eq:freq:def-of-cell}
	E = E(r,m, u^{(1)},u^{(2)},\dots, u^{(r-1)},v),
	\end{align}
	where $r,m \in \NN_0,\ u^{(1)},u^{(2)},\dots, u^{(r-1)},v \in \Sigma_2^*$, 
	defined as the set of those $n \in [N]$ for which $r(n) = r$, $\absbig{u^{(r)}_{n}} = m$, $u^{(1)}_n = u^{(1)}$, $u^{(2)}_n = u^{(2)}$, \dots, $u^{(r-1)}_n = u^{(r-1)}$, and $v_n = v$.
	We will say that the cell $E$ given by \eqref{eq:freq:def-of-cell} is ``good'' if all of the following conditions hold, and ``bad'' otherwise:
\begin{enumerate}
\item\label{it:basic:147:A} $\displaystyle \absnormal{ r - \sqrt{L}} \leq \sqrt{ 2\log L}$;
\item\label{it:basic:147:B} $\displaystyle \abs{ m - 2r } \leq \sqrt{20 r \log L}$;
\item\label{it:basic:147:C} $\displaystyle [u^{(1)}u^{(2)}\dots, u^{(r-1)}1^{m}v]_2 < N$.
\end{enumerate} 
Let also $U_{\good}$ and $U_{\bad}$ denote the union of all good and bad cells, respectively. Our plan is to show that $f$ is well-behaved on $U_{\good}$ and that $U_{\bad}$ is negligibly small.

We begin with estimating the cardinality of $U_{\bad}$.
Let $Z$ denote the set of those $n \in [N]$ whose binary expansion $(n)_2^L = 0 \dots 0(n)_2$ (padded with $0$s to length $L$) contains a subword $w$ with 
\begin{align}\label{eq:473:2}
\abs{ \abs{w}_1 - \frac{1}{2}\abs{w} } > \sqrt{ \frac{3}{2} \abs{w} \log L} .
\end{align}
It follows from Hoeffding's inequality that, if we fix the positions where $w$ begins and ends and we choose $n \in [2^L]$ uniformly at random, the probability that \eqref{eq:473:2} holds at most $2 \exp(-3 \log L) = 2 L^{-3}$. Since there are $\binom{L+1}{2}$ ways to pick a non-empty subword of a length-$L$ word, using the union bound we obtain
\[
	\# Z \leq 2^L \cdot \binom{L+1}{2} \cdot 2 L^{-3} \ll \frac{N}{L}.
\]
Consider a bad cell $E$, given by \eqref{eq:freq:def-of-cell}, and assume that $N$ (and hence also $L$) is sufficiently large.
If \ref{it:basic:147:A} is false then $E \subset Z$, since for any $n \in E$, condition \eqref{eq:473:2} holds for $w = (n)_2^L$, which can easily be derived from the fact that $\abs{w} = L$ and $(r^2+r)/2 \leq \abs{w}_1 \leq (r^2+3r)/2$.
Similarly, if  \ref{it:basic:147:A} is true but \ref{it:basic:147:B} is false then $E \subset Z$, since for any $n \in E$, condition \eqref{eq:473:2} holds for $w = u^{(r(n))}_n$. Finally, if \ref{it:basic:147:A} is true but \ref{it:basic:147:C} is false then the word $u^{(1)}u^{(2)}\dots, u^{(r-1)}$ is a prefix of $(N)_2$, and its length is at least $r(r-1)/2 \geq L/3$. Thus, $E$ is contained in an interval of length $2^{2/3 L} \ll N^{2/3} < N/L$. Combining the estimates obtained above, we conclude that 
\begin{align}\label{eq:freq:E_bad-is-small}
\# U_{\bad} \ll N/L  = o_{N \to \infty}(N).
\end{align}

Consider now a ``good'' cell $E$, given by \eqref{eq:freq:def-of-cell}. It follows from condition \ref{it:basic:147:C} that
\begin{align}\label{eq:freq:form-of-E}
	E = \set{ [u^{(1)} u^{(2)} \dots u^{(r-1)} w v ]_2 }{w \in \Sigma_2^m,\ \abs{w}_1 = r}.
\end{align}
Thus, it follows from \eqref{eq:freq:def-of-f} that
\begin{align}\label{eq:473:8}
	\#\set{ n \in E}{ a_n  = 1} = \ip{\theta \binom{m}{r}} = \ip{ \theta \cdot \# E}.
\end{align}
Assuming that $L$ is sufficiently large, we have
\begin{align}\label{eq:473:7}
\binom{m}{r} \geq \bra{ \frac{m}{r} }^r \geq \bra{\frac{3}{2}}^{\sqrt{L}/2} = \omega_{N \to \infty}(1).
\end{align}
(We recall that $\omega_{N \to \infty}(1)$ denotes an expression that tends to $\infty$ with $N$.) Thus, the expression in \eqref{eq:473:8} is $(\theta - o_{N\to \infty}(1)) \# E$.

We are now ready to estimate
\begin{align*}
	\# \set{ n \in [N] }{ a_n = 1} &\geq \# \set{ n \in U_{\good} }{ a_n = 1} 
	 = \sum_{E \text{ good} } (\theta - o_{N\to \infty}(1)) \cdot \# E
	 \\ &= (\theta - o_{N\to \infty}(1)) (N-\# U_{\bad}) = \theta N  - o_{N\to \infty}(N).
\end{align*}
where the sum is taken over all the cells $E$ that are ``good''. Dividing by $N$ and letting $N \to \infty$, we conclude that $\freqinf_{\bb a}(1) \geq \theta$. Using analogous reasoning, we obtain the complementary estimate $\freqsup_{\bb a}(1) \leq \theta$. \end{claimproof}	

Combining Claims \ref{claim:freq:A} and \ref{claim:freq:B} completes the argument.
\end{proof} 
\section{Subword complexity}\label{sec:subword}\color{double}

Subword complexity of a sequence $\bb a$ over a finite alphabet $\Omega$ is the function $p_{\bb a}$ which assigns to each $\ell \in \NN$ the number of length-$\ell$ subwords of $\bb a$:
\begin{equation}\label{eq:subword:def-p_a}
	p_{\bb a}(\ell) := \# \set{w \in \Omega^\ell}{ (\exists \ n \in \NN_0) \ w =  a_{[n,n+\ell)} }.
\end{equation}
A notable feature of automatic sequences is that their subword complexity is linear: $p_{\bb a}(\ell) = O(\ell)$. This property does not directly extend to asymptotically automatic sequences. In fact, given a finite alphabet $\Omega$ it is not hard to construct a sequence over $\Omega$ that is asymptotically constant (and thus asymptotically $k$-automatic for each $k \geq 2$) and which contains each finite word as a subword (and thus has maximal possible subword complexity, $p_{\bb a}(\ell) = \# \Omega^{\ell}$). In particular, we see that two asymptotically equal sequences can have diametrically different subword complexities.

The discussion above suggests that subword complexity is not the right notion to consider in the asymptotic regime. Instead, we introduce a new notion of \emph{asymptotic subword complexity}, defined as the function $\tilde p_{\bb a}$ which assigns to each $\ell \in \NN$ the number of length-$\ell$ subwords of $\bb a$ which appear with positive asymptotic frequency:
\begin{align}\label{eq:subword:def-tilde-p_a}
	\tilde p_{\bb a}(\ell) &:= \# \set{w \in \Omega^\ell}{ \freqsup_{\bb a}(w) > 0}.
\end{align}

It follows directly from the definition that for any sequence $\bb a$ and any $\ell \in \NN$ we have $\tilde p_{\bb a}(\ell) \leq p_{\bb a}(\ell)$. Moreover, $p_{\bb a}(\ell)$ only depends on the equivalence class of $\bb a$ up to asymptotic equality, that is, if $\bb a \simeq \bb b$ then $\tilde p_{\bb a}(\ell) = \tilde p_{\bb b}(\ell)$ for all $\ell \in \NN$.

\begin{proposition}\label{prop:subword:auto}
	Let $\bb a$ be an asymptotically automatic sequence. Then $\tilde p_{\bb a}(\ell) = O(\ell)$.
\end{proposition}
\begin{proof}
	Assume that $\bb a$ is asymptotically $k$-automatic and takes values in a finite alphabet $\Omega$. Pick $d$, $\bb a^{(0)}, \bb a^{(1)},\dots, \bb a^{(d-1)}$ and $\phi \colon \Sigma_k^* \to \Sigma_{d}$ as in Proposition \ref{prop:basic:structure}.
	
		Let $\ell \in \NN$ and pick $i \in \NN$ with $k^{i-1} \leq \ell < k^i$. Consider the sequence $\bb a'$ given by
	\[
		a'_n = a_{\ip{n/k^i}}^{(\phi((n)_k^i))}.
	\]
	Since $\bb a \simeq \bb a'$, we have $\tilde p_{\bb a'}(\ell) = \tilde p_{\bb a}(\ell)$. Thus, it suffices to show that there are $O(\ell)$ length-$\ell$ words which appear in $\bb a'$ with positive upper frequency. Pick one such word $w \in \Omega^\ell$. Applying the pigeonhole principle, we can find $0 \leq r < k^i$ such that
	\[
		\overline{d}\bra{\set{n \in \NN_0}{ a'_{[k^i n + r, k^i n + r + \ell)} = w} } > 0.
	\]
	Recall that for $h \in [0,2k^i)$ we have 
	\[ a'_{k^i n + h} = a_{n+\ip{h/k^i}}^{(\phi((h)_k^i))},\]
	where $\ip{h/k^i} \in \{0,1\}$. Thus, the word $a'_{[k^i n + r, k^i n + r + \ell)}$ is completely determined by $\brabig{ a^{(j)}_n }_{j=0}^{d-1}$, $\brabig{ a^{(j)}_{n+1} }_{j=0}^{d-1}$ and $r$. It follows that
	\[
		\tilde p_{\bb a'}(\ell) \leq \#\Omega^{d} \cdot \#\Omega^{d} \cdot k^i \leq C \ell,
	\]
	where $C := k \cdot \#\Omega^{2d}$. In particular, $\tilde p_{\bb a'}(\ell) = O(\ell)$, as needed.
\end{proof}

Although the notion of asymptotic subword complexity is quite natural, it does not seem to have been systematically studied in the literature. To justify interest in this notion, we prove several basic properties, which are analogues of standard facts about subword complexity. Given a sequence $\bb a \in \Omega^\infty$, it will be convenient to consider the set given by
\begin{align}\label{eq:subword:def-X_a}
	\tilde X_{\bb a} := \set{\bb x \in \Omega^\infty}{ \freqsup_{\bb a}(w) > 0 \text{ for each subword $w$ of $\bb x$} }.
\end{align}
It is routine to verify that for a word $w \in \Omega^*$ we have $\freqsup_{\bb a}(w) > 0$ if and only $w$ is a subword of some $\bb x \in \tilde X_{\bb a}$. In particular, $\tilde p_{\bb a}(\ell) \geq p_{\bb x}(\ell)$ for all $\bb x \in \tilde X_{\bb a}$ and $\ell \in \NN$. 

\begin{lemma}\label{lem:subword:decompose}
	Let $\bb a$ be a sequence over a finite alphabet $\Omega$ such that $\freqsup_{\bb a}(\omega) > 0$ for all $\omega \in \Omega$. Then exactly one of the following holds:
	\begin{enumerate}
	\myitem{$(\dagger)$}\label{it:subword:non-greedy} there exist $w \in \Omega^*$ and $\bb x \in \tilde X_{\bb a}$ such that  $\bb a = w \bb x$;
	\myitem{$(\ddagger)$}\label{it:subword:greedy} there exist $w_i \in \Omega^*$ such that $\freqsup_{\bb a}(w_i) > 0 = \freqsup_{\bb a}(w_i (w_{i+1})_1)$ for all $i \in \NN_0$ and $\bb a = w_0 w_1 w_2  \dots$.
	\end{enumerate}
\end{lemma}
\begin{proof}
	We first note that if $\bb a$ admits a representation as in \ref{it:subword:non-greedy} then $\freqsup_{\bb a}(u) > 0$ for each subword $u$ of $\bb x$ and consequently $\bb a$ cannot be represented as in \ref{it:subword:greedy}.

	Assume next that $\bb a$ cannot be represented as in \ref{it:subword:non-greedy}. We greedily construct a representation as in \ref{it:subword:greedy}. Let $\a_0$ denote the initial symbol of $\bb a$. Note that $\bb a$ has a prefix $w$ with $\freqsup_{\bb a}(w) = 0$ since otherwise we would have $\bb a \in \tilde X_{\bb a}$. Since $\freqsup_{\bb a}(\omega) > 0$ for $\omega \in \Omega$, we have $\abs{w} \geq 2$. Pick the shortest possible $w$ and write it in the form $w = w_0 \alpha_1$ where $w_0 \in \Omega^*$ and $\alpha_1 \in \Omega$. By definition, $\freqsup_{\bb a}(w_0) > 0$. We may write $\bb a$ in the form $\bb a = w_0 \bb a'$, where $\bb a' \in \Omega^\infty$ begins with $\a_1$. Then $\bb a'$ has a minimal prefix, which we will denote by $w_1 \a_2$, with $\freqsup_{\bb a}(w_1\a_2) = 0$, since otherwise we would have $\bb a \in w_0 \tilde X_{\bb a}$. Thus, we may write $\bb a' = w_1 \bb a''$. 
	Iterating this reasoning, we find $w_0,w_1,w_2,\dots \in \Omega^*$ and $\a_0,\a_1,\a_2,\dots \in \Omega$ such that $\bb a = w_0w_1w_2\dots$ and for each $i \in \NN_0$, $\a_i$ is the initial symbol of $w_i$, $\freqsup_{\bb a}(w_i) > 0$ and 
	$\freqsup_{\bb a}(w_i \a_{i+1}) = 0$. 
\end{proof}

We point out that the assumptions of Lemma \ref{lem:subword:decompose} are easy to satisfy. Indeed, if it is not the case that  $\freqsup_{\bb a}(\omega) > 0$ for all $\omega \in \Omega$ then we can find a smaller alphabet $\Omega' \subset \Omega$ and a sequence $\bb a' \in (\Omega')^\infty$ with $\bb a' \simeq \bb a$ and $\freqsup_{\bb a'}(\omega) > 0$ for all $\omega \in \Omega'$. 

\begin{lemma}\label{lem:subword:w-is-long}
	Let $\bb a$ be a sequence over a finite alphabet $\Omega$ such that $\freqsup_{\bb a}(\omega) > 0$ for all $\omega \in \Omega$, and assume that $\bb a$ admits a representation as in \ref{it:subword:greedy}. Then
\begin{equation}\label{eq:subword:85:1}
	\lim_{N \to \infty} \frac{1}{N} \sum_{n=0}^{N-1} \abs{w_n} = \infty.
\end{equation}	 
\end{lemma}
\begin{proof}
	For the sake of contradiction, suppose that
\begin{equation}\label{eq:subword:85:2}
	\liminf_{N \to \infty} \frac{1}{N} \sum_{n=0}^{N-1} \abs{w_n} < L
\end{equation}		
for some integer $L$. Let $N$ be a large integer such that 
\begin{equation}\label{eq:subword:85:3}
	\sum_{n=0}^{N-1} \abs{w_n} < NL,
\end{equation}	
and let $\cM(N)$ denote the set
\begin{equation}\label{eq:subword:85:4}
	\cM(N) = \set{ n \in [N] }{ \abs{w_n} \leq 2L}.
\end{equation}	
It follows from Markov's inequality that $\# \cM(N) \geq N/2$. For $w \in \Omega^*$ with $\abs{w} \leq 2L$ and $\alpha \in \Omega$ such that $\freqsup_{\bb a} (w \a) = 0$, 
consider
\[
	\cM(N;w,\alpha) = \set{n \in \cM(N) }{ w_n = w,\ w_{n+1} \text{ begins with } \alpha}.
\]
Note that $[N] = \bigcup_{w,\a} \cM(N;w,\alpha)$, where the union is taken over all pairs $w,\a$ which satisfy the conditions mentioned above. If $\# \Omega = 1$ there is nothing to prove, so we may freely assume that $\# \Omega \geq 2$. The number of possible pairs $w,\a$ is at most $\# \Omega \cdot (\#\Omega^{2L} + \#\Omega^{2L-1} + \dots + 1) < \#\Omega^{2L+2}$, so we can find a pair $w,\alpha$ with $\# \cM(N;w,\alpha) \geq N/2 \#\Omega^{2L+2}$.

Since there are finitely many possible pairs $w,\alpha$, we conclude that for one of them we have $\# \cM(N;w,\alpha) \geq N/2 \#\Omega^{2L+2}$ for infinitely many values of $N$. In particular, we have
\begin{equation}\label{eq:subword:85:5}
	\freqsup_{\bb a} (w \a) \geq \limsup_{N \to \infty} \frac{\# \cM(N; w,\a) }{\sum_{n=0}^{N-1} \abs{w_n} } \geq \frac{1}{2L \#\Omega^{2L+2} } > 0,
\end{equation}
contradicting an earlier assumption.
\end{proof}

We are now ready to classify sequences with particularly low asymptotic subword complexity. Recall that each sequence $\bb a$ such that $p_{\bb a}(\ell) \leq \ell$ for at least one $\ell \in \NN$ is ultimately periodic, in which case $p_{\bb a}(\ell)$ is bounded as $\ell \to \infty$ \cite{CovenHedlund-1973}. A similar result is true in the asymptotic regime. Recall that we say that a sequence $\bb a$ is asymptotically invariant under a shift if there exists $m \in \NN$ such that $(a_n)_{n=0}^\infty \simeq (a_{n+m})_{n=0}^\infty$.

\begin{proposition}\label{prop:subword:periodic}
	Let $\bb a$ be a sequence over a finite alphabet $\Omega$. If there exists $\ell \in \NN$ such that $\tilde p_{\bb a}(\ell) \leq \ell$ then $\bb a$ is asymptotically invariant under a shift. Conversely, if $\bb a$ is asymptotically invariant under a shift then $\tilde p_{\bb a}(\ell)$ is bounded as $\ell \to \infty$.
\end{proposition}
\begin{proof}
	We may assume without loss of generality that $\freqsup_{\bb a}(\omega) > 0$ for all $\omega \in \Omega$. For each $\bb x \in \tilde X_{\bb a}$ we have $p_{\bb x}(\ell) \leq \tilde p_{\bb a}(\ell) \leq \ell$, and hence $\bb x = v u^\infty$ is eventually periodic with the pre-period $\abs{v}$ and the period $\abs{u}$ bounded as a function of $\ell$ and $\# \Omega$, say $\abs{v},\abs{u} \leq L$ (the bound on the period can be inferred e.g.\ from the argument given in \cite[Prop.\ 1.1.1]{Fogg-book}). It now follows from Lemma \ref{lem:subword:w-is-long} that $\bb a$ is asymptotically invariant under the shift by $L!$. Conversely, if $\bb a$ is asymptotically invariant under the shift by some $m \in \NN$ then $\tilde p_{\bb a}(\ell) \leq \# \Omega^m$ for all $\ell \in \NN$. 
\end{proof}

It follows from the remark above that for each sequence $\bb a$ that is not eventually periodic we have $p_{\bb a}(\ell) \geq \ell+1$ for all $\ell \in \NN$. Sequences for which equality $p_{\bb a}(\ell) = \ell+1$ holds for all $\ell \in \NN$ are known as Sturmian sequences and have been extensively studied; see e.g.\ \cite[Chpt.\ 5]{Fogg-book} for extensive discussion. There are several alternative characterisations of Sturmian sequences. Specifically, a sequence $\bb a$ is Sturmian if and only if it takes one of the following forms:
\begin{align}\label{eq:subword:def-sturmian}
	a_n & = \ip{\theta (n+1) + \rho} - \ip{\theta n + \rho}, 
	& \text{or} &&
	a_n & = \ceil{\theta (n+1) + \rho} - \ceil{\theta n + \rho}, 
\end{align}
where $\theta \in (0,1) \setminus \QQ$ and $\rho \in [0,1)$. We have the following asymptotic analogue. (Recall that by ``for almost all'' we mean ``for all outside of a set with density zero''.)

\begin{proposition}\label{prop:subword:sturmian}
	Let $\bb a$ be a sequence over $\{0,1\}$ and assume that $\tilde p_{\bb a}(\ell) = \ell+1$ for all $\ell \in \NN$. Then there exists $\theta \in [0,1) \setminus \QQ$ and $\bm \rho \in [0,1)^\infty$ such that $\rho_{n+1} = \rho_{n}$ for almost all $n \in \NN_0$ and $a_n = \ip{(n+1)\theta+\rho_n} - \ip{n \theta + \rho_n}$ for all $n \in \NN_0$.
\end{proposition}
\begin{proof}
	We may assume without loss of generality that $\freqsup_{\bb a}(\omega) > 0$ for $\omega \in \{0,1\}$. If $\bb a$ admits a decomposition $\bb a = w \bb x$ as in \ref{it:subword:non-greedy} then $\tilde p_{\bb a}(\ell) = p_{\bb x}(\ell) = \ell+1$ for all $\ell \in \NN$, and thus $\bb x$ is a Sturmian word. As a consequence, we can find the required representation of $\bb a$ with $\bm \rho$ eventually constant.
	Thus, we may assume that $\bb a$ does not admit a decomposition  \ref{it:subword:non-greedy} and hence it has a decomposition  \ref{it:subword:greedy} for some words $w_0,w_1,w_2,\dots$.

	For all $\bb x \in \tilde X_{\bb a}$ we have $p_{\bb x}(\ell) \leq \tilde p_{\bb a}(\ell) = \ell +1$ for all $\ell \in \NN$. Thus, all $\bb x \in \tilde X_{\bb a}$ are either Sturmian or eventually periodic. Suppose first that $\tilde X_{\bb a}$ contains at least one Sturmian word $\bb x$. Let $\theta \in [0,1) \setminus \QQ$ denote the slope of $\bb x$; thus, $\bb x$ is given by $x_n = \ip{(n+1)\theta+\rho} - \ip{n \theta + \rho}$ ($n \in \NN_0$) for some $\rho \in [0,1)$ (possibly with the exception of one value of $n$). For each $\ell \in \NN$ we have $\tilde p_{\bb a}(\ell) = p_{\bb x}(\ell) = \ell+1$. Hence, each word $w$ with $\freqsup_{\bb a}(w) > 0$ is a subword of $\bb x$.  In particular, for each $i \in \NN_0$, there is some $\sigma_i \in [0,1)$ such that $w_{i,m} = \ip{(m+1)\theta+\sigma_i} - \ip{m \theta + \sigma_i}$ for $0 \leq m < \abs{w_i}$. It follows that we can pick $\bm \rho$ of the form 
	\[
		\bm \rho = 
		\underbrace{\rho_0'\rho_0' \dots \rho_0'}_{\abs{w_0} \text{ times}}
		\underbrace{\rho_1'\rho_1' \dots \rho_1'}_{\abs{w_1} \text{ times}}
		\underbrace{\rho_2'\rho_2' \dots \rho_2'}_{\abs{w_2} \text{ times}}
		\dots
	\]
	for some $\rho_1',\rho_2',\dots \in [0,1)$. As a consequence of Lemma \ref{lem:subword:w-is-long} we have that $\rho_{n+1} = \rho_n$ for almost all $n \in \NN_0$.
	
	Now, suppose that all $\bb x \in \tilde X_{\bb a}$ are eventually periodic, and hence can be written in the form $\bb x = v u^\infty$ for some $v \in \Omega^*$ and $u \in \Omega^*$ with $u \neq \epsilon$ which cannot be written as a power of a shorter word. Our next goal is to show that the lengths of the periods $\abs{u}$ are uniformly bounded. 
	
	For the sake of contradiction, assume that for each $\ell_0$ there is a word $\bb x = v u^\infty \in \tilde X_{\bb a}$ as above with $\ell := \abs{u} \geq \ell_0$. Note that $p_{\bb x}(\ell-1) = \ell$. Indeed, $p_{\bb x}(\ell-1)$ cannot be larger since $p_{\bb x}(\ell-1) \leq \tilde p_{\bb a}(\ell-1) = \ell$, and it cannot be smaller since that would imply that $u^\infty$ has a period strictly less than $\ell$. Thus, each word $w \in \Omega^{\ell-1}$ with $\freqsup_{\bb a}(w) > 0$ is a subword of $u^\infty$. Let us next consider a pair of sequences $v_1 u_1^\infty,\ v_2 u_2^\infty  \in  \tilde X_{\bb a}$ with $\abs{u_1} =: \ell_1 < \ell_2 := \abs{u_2}$. Since $\freqsup_{\bb a}(u_1^{\ell_2/\ell_1}) > 0$, up to periodic shift and up to changing at most one symbol, the words $u_1^{\ell_2/\ell_1}$ and $u_2$ are equal. Consider now a third sequence $v_3 u_3^\infty  \in \tilde X_{\bb a}$ with $\ell_3 := \abs{u_3} > 100 \ell_1 \ell_2$. Applying the earlier argument to the pairs $v_1 u_1^\infty,\ v_3 u_3^\infty$ and $v_2 u_2^\infty,\ v_3 u_3^\infty$ we conclude that there exist periodic shifts $\tilde u_1$ and $\tilde u_2$ of $u_1$ and $u_2$ such that $\tilde u_1^{\ell_3/\ell_1}$ and $\tilde u_2^{\ell_3/\ell_2}$ differ from $u_3$ in at most one position. Bearing in mind that $\ell_1\ell_2$ is a period of both $\tilde u_1^\infty$ and $\tilde u_2^\infty$, we conclude that $\tilde u_1^\infty = \tilde u_2^\infty$, contradicting the assumption that $\ell_1 < \ell_2$.

	Let $L$ be the maximal period of $\bb x \in \tilde X_{\bb a}$ and put $M = L!$. Note that it follows from Proposition \ref{prop:subword:periodic} that $\bb a$ is not asymptotically invariant under the shift by $M$. Hence, there exists $\e > 0$ such that we can find arbitrarily long words $w \in \Omega^*$ with $\freqsup_{\bb a}(w) > 0$ such that there are at least $\e \abs{w}$ positions $0 \leq n < \abs{w} - M$ such that $w_n \neq w_{n+M}$ (this follows readily from Lemmas \ref{lem:subword:decompose} and \ref{lem:subword:w-is-long}). For each word $w \in \Omega^*$ with $\freqsup_{\bb a}(w) > 0$ and $\abs{w} \geq M$, let $\bar w$ denote the word over $\{0,1\}$ with length $\abs{\bar w} = \abs{w} - M$ given by $\bar w_n = 0$ if $w_{n} = w_{n+M}$ and $w_n = 1$ otherwise, and let $\cL$ be the family of all words $\bar w$ that arise this way. With the same $\e > 0$ as above, we see that there exist arbitrarily long words $\bar w \in \cL$ such that $\abs{\bar w}_1 \geq \e \abs{w}$. Thus, it follows from Lemma \ref{lem:subword:compactness} below that there exists $\bar{\bb x} \in \{0,1\}^\infty$ such that each subword of $\bar{\bb x}$ belongs to $\cL$ and $\freqsup_{\bar{\bb x}}(1) > 0$. A compactness argument shows that there exists $\bb x \in \tilde X_{\bb a}$ such that for each $n \in \NN_0$ we have $\bar x_n = 0$ if $\bar x_n = x_{n+M}$ and $x_n = 1$ otherwise. This contradicts the earlier observation that each $\bb x \in \tilde X_{\bb a}$ is eventually periodic with period at most $L$, and finishes the argument.	
\end{proof}

\begin{lemma}\label{lem:subword:compactness}
Let $\cL \subset \{0,1\}^*$ be an infinite set of words closed under taking subwords. Suppose that there exists $\e > 0$ such that for infinitely many words $w \in \cL$ we have $\abs{w}_1 \geq \e \abs{w}$. Then there exists $\bb x \in \{0,1\}^\infty$ such that $\freqsup_{\bb x}(1) > 0$ and each subword of $\bb x$ belongs to $\cL$.
\end{lemma}
\begin{proof}
	Let $\a$ denote the infimum of all real numbers $\e > 0$ such that for each sufficiently long $w \in \cL$ we have $\abs{w}_1 < \e \abs{w}$. By assumption, $\a > 0$. Pick a decreasing sequence $\delta_\ell > 0$ such that $\delta_\ell \to 0$ as $\ell \to \infty$ and $\abs{w}_1 < (\a + \delta_\ell) \abs{w}$ for all $w \in \Sigma_k^\ell$. Let $(m_i)_{i=1}^\infty$ be a rapidly increasing sequence of positive integers, to be determined in the course of the argument, and let $(\ell_i)_{i=0}^\infty$ be the sequence given by $\ell_0 = 1$ and $\ell_{i+1} = m_{i+1} \ell_i$.

	Consider, for some $i \in \NN$, a word $w \in \cL$ of length $\ell_{i}$ satisfying $\abs{w}_1 \geq (\a - \delta'_{\ell_i})\ell_{i}$, where $\delta_{\ell_i}' := (4/\a - 1)\delta_{\ell_i}	$. We may decompose $w = w_0 w_1 \dots w_{m_i-1}$ where $\abs{w_r} = \ell_{i-1}$ for all $0 \leq r < m_i$. Expressing $\abs{w}_1$ as $\sum_{r=0}^{m_i-1} \abs{w_r}_1$, we that the proportion of words  $w_r$ ($ 0 \leq r < m_i$) such that $\abs{w_r}_1 < (\a - \delta'_{\ell_{i-1}})\ell_{i-1}$ is at most 
	\begin{align}\label{eq:530:1}
		\frac{\delta_{\ell_{i-1}}+\delta_{\ell_i}'}{\delta_{\ell_{i-1}}' + \delta_{\ell_{i-1}}} = \frac{\a}{4} + \delta_{\ell_i} \cdot \frac{4/\a-1}{\delta_{\ell_{i-1}}' + \delta_{\ell_{i-1}}}.
	\end{align}
	Picking $m_{i}$ sufficiently large as a function of $\ell_{i-1}$, we can ensure that the expression in \eqref{eq:530:1} is less than $\a/2$. 
	As a consequence we can decompose $w$ as
	\(
		w = u w' v,
	\)		
	where $\abs{w'} = \ell_{i-1}$, $\abs{w'}_1 \geq (\alpha-\delta'_{\ell_{i-1}})\ell_{i-1}$ and $\abs{v}_1 \geq (\alpha/2)\ell_i$ (we take $w' = w_r$ for the minimal admissible value of $r$).
	Iterating this construction, we can find words $w^{(j)},u^{(j)},v^{(j)} \in \cL$ for $0 \leq j \leq i$ such that $w^{(i)} = w$ and $w^{(j+1)} = u^{(j)} w^{(j)} v^{(j)}$, $\absnormal{w^{(j)}} = \ell_{j}$, $\absnormal{w^{(j)}}_1 \geq (\alpha-\delta'_{\ell_{j}})\ell_{j}$ and $\abs{v^{(j)}}_1 \geq (\alpha/2)\ell_{j+1}$ for $0 \leq j < i$. Applying a compactness argument, we can find infinite sequences $w^{(j)},u^{(j)},v^{(j)} \in \cL$ which satisfy the properties mentioned above for all $j \in \NN_0$.
	
	Consider the sequence $\bb x = v^{(0)}v^{(1)}v^{(2)}\dots$. It is clear from the construction that all subwords of $\bb x$ belong to $\cL$. Since $\absnormal{v^{(j)}}_1 \geq (\alpha/2)\ell_{j+1}$ and $\sum_{h=0}^{j} \absnormal{v^{(h)}} < \ell_{j+1}$, we have $\freqsup_{\bb x}(1) \geq \a/2 > 0$, as needed.
\end{proof} 

\color{double}
\section{Classification problems}\label{sec:class}

Given a class of sequences $\mathcal{C}$, a natural question in the theory of automatic sequences is: Which members of $\mathcal{C}$ are automatic? Here, we discuss some problems of this type in the asymptotic regime.

\subsection{Bracket sequences}
A generalised polynomial is an expression constructed using ordinary polynomials, the floor function, addition and multiplication. In contrast with ordinary polynomials, their generalised counterparts can be bounded or even finitely-valued but non-constant. Bounded generalised polynomials have been extensively investigated by many authors, including H{\aa}land--Knutson, Bergelson, Leibman and others, see e.g. \cite{Haland-1993}, \cite{Haland-1994}, \cite{HalandKnuth-1995}, \cite{BergelsonLeibman-2007}, \cite{Leibman-2012}, \cite{BergelsonHalandSon-2020}. In \cite{AdamczewskiKonieczny} we undertook a systematic study of letter-to-letter codings of finitely-valued generalised polynomial sequences, which we dub bracket sequences. In a series of papers, \cite{ByszewskiKonieczny-2018-TAMS}, \cite{ByszewskiKonieczny-2020-CJM}, \cite{Konieczny-2021-JLM}, we obtained a complete classification of automatic bracket sequences: they are precisely the eventually periodic sequences. The vast bulk of the difficulty in the aforementioned problem stems from the need to consider asymptotically constant sequences. This issue does not arise when we consider asymptotically automatic sequences, leading to a much shorter argument. 

\begin{proposition} 
	Let $\bb a$ be a bracket word over a finite alphabet. Then the following conditions are equivalent:
\begin{enumerate}
\item\label{it:class:75:A} $\bb a$ is asymptotically $k$-automatic for at least one $k \geq 2$;
\item\label{it:class:75:B} $\bb a$ is asymptotically $k$-automatic for all $k \geq 2$;
\item\label{it:class:75:C} $\bb a$ is asymptotically equal to a periodic sequence.
\end{enumerate}
\end{proposition}
\begin{proof} 
	It is clear that we have the chain of implications \ref{it:class:75:C} $\Rightarrow$ \ref{it:class:75:B} $\Rightarrow$ \ref{it:class:75:A}, so we only need to prove \ref{it:class:75:A} $\Rightarrow$ \ref{it:class:75:C}. Since the argument closely follows the lines of the proof of Theorem B in \cite{ByszewskiKonieczny-2020-CJM} we will only provide a sketch of the argument with emphasis on the points where the two arguments differ, and refer to  \cite{ByszewskiKonieczny-2020-CJM} for details.
	
	In \cite{ByszewskiKonieczny-2020-CJM} we use the notion of a weakly periodic function $f \colon \NN_0 \to \Omega$ ($\Omega$ being any set). A function $f$ is weakly periodic if for any restriction $f'$ of $f$ to an arithmetic sequence (given by $f'(n) = f(an+b)$ with $a \in \NN$ and $b \in \NN_0$) there exist $q \in \NN$ and distinct $r,r' \in \NN_0$ such that 
\begin{align}\label{eq:class:747:1}
f'(qn + r) &= f'(qn+r') &\text{for all $n \in \NN_0$}.
\end{align}
Here, we will use an even weaker notion of an asymptotically weakly periodic function, which is defined in the same way with the exception that in \eqref{eq:class:747:1} we only require the equality to hold for almost all $n \in \NN_0$. 

Lemma 2.1 in \cite{ByszewskiKonieczny-2020-CJM} asserts that each automatic sequence is weakly periodic. With essentially the same argument we conclude that each asymptotically automatic sequence is asymptotically weakly periodic. (The only difference is that in the asymptotic variant, we identify sequences that are equal almost everywhere.)

We will show, in analogy with Theorem 2.6 in \cite{ByszewskiKonieczny-2020-CJM}, that if $f \colon \NN_0 \to \RR$ is a finitely-valued generalised polynomial sequence that is asymptotically weakly periodic then $f$  is asymptotically equal to a periodic sequence. Combined with the aforementioned analogue of \cite[Lemma 2.1]{ByszewskiKonieczny-2020-CJM}, this will complete the proof of the implication \ref{it:class:75:A} $\Rightarrow$ \ref{it:class:75:C}. 

We rely on the structure theorem due to Bergelson and Leibman \cite{BergelsonLeibman-2007} (cf.\ \cite[Theorem 1.13]{ByszewskiKonieczny-2020-CJM}) which implies that there exists a minimal nilsystem $(X,T)$, a point $z \in X$, a semialgebraic partition $X = \bigcup_{j=1}^r S_j$ and constants $c_j \in \RR$ such that $g(n) = c_j$ if and only if $T^n(z) \in S_j$ ($1 \leq j \leq r$). (Nilsystem is a special type of a topological dynamical system, i.e., a compact space $X$ equipped with a homeomorphism $T \colon X \to X$. Minimality means that each orbit $\set{T^n(x)}{n \in \NN}$ is dense in $X$. There is a natural choice of a $T$-invariant measure $\mu_X$ on $X$, namely the corresponding Haar measure. For our purposes, all that we need to know about the sets $S_j$ is that $\mu_X(\partial S_j) = 0$, where $\partial S = \operatorname{cl} S \setminus \operatorname{int} S$ denotes the boundary of a set $S$.)
Passing to an arithmetic progression (cf.\ \cite[Remark 1.14]{ByszewskiKonieczny-2020-CJM}) we may assume that $(X,T)$ is totally minimal, meaning that $(X,T^a)$ is minimal for each $a \in \NN$. (We note that this reduction is possible because we are working with nilsystems rather than more general topological dynamical systems.)

Replacing $c_j$ with $0$s and $1$s, we may freely reduce the problem to the simpler situation where $f \colon \NN_0 \to \{0,1\}$ and $f(n) = 1$ if and only if $T^n(z) \in S$ for some semialgebraic set $S$, which is reminiscent of the situation in Lemma 2.4 in \cite{ByszewskiKonieczny-2020-CJM}. Our goal is to show that $f$ is asymptotically equal to a constant sequence. Replacing $S$ with its interior, we may freely assume that $S$ is open. (Corollary 1.12 in \cite{ByszewskiKonieczny-2020-CJM} guarantees that this operation only changes $f$ on a density zero set of positions, and hence does not affect asymptotic weak periodicity nor the property of being asymptotically constant.) If $S$ is empty then there is nothing to prove, so assume that this is not the case.

Since $f$ is asymptotically weakly periodic, we can find $q,r,r' \in \NN_0$ with $q \neq 0$ and $r < r'$ such that $f(qn+r) = f(qn+r')$ for almost all $n$. Let $d = r' - r$. We claim that $T^d(S) \subset S$. Pick $y \in S$ an open neighbourhood $V$ of $T^d(y)$; it will suffice to show that $V \cap S \neq \emptyset$. Let $U \subset S$ be an open neighbourhood of $y$ with $T^d(U) \subset V$, and let $\cN = \set{ n \in \NN_0}{T^{qn+r}(z) \in U}$. Since $(X,T)$ is minimal, the set $\cN$ is syndetic. For almost all $n \in \cN$ (thus, in particular, for at least one $n$) we have $f(qn+r') = f(qn+r) = 1$ and consequently $T^{qn+r'}(z) = T^d(T^{qn+r}(z)) \in U$. Thus, $V \cap S \supseteq T^d(U) \cap S \neq \emptyset$, as needed.

Since $T^d(S) \subset S$, the orbit of any point from $S$ under $T^d$ is contained in $S$. Since $(X,T^d)$ is minimal, it follows that $S$ is dense in $X$. Recalling that $\mu_X(\partial S) = 0$, we conclude that $\mu_X(S) = 1$ and hence $f(n) = 1$ for almost all $n \in \NN_0$ (cf.\ Corollary 1.12 in \cite{ByszewskiKonieczny-2020-CJM}).
\end{proof}

{
As an application, we can consider bracket sequences of particularly simple form and classify the ones that are asymptotically equal to an automatic sequence, cf.\ \cite[Thm.\ 6.2]{AlloucheShallit-2003}.
\begin{corollary}
	Let $\a,\b \in \RR$, $m \in \NN_{\geq 2}$, and let $\bb a$ be the sequence given by $a_n = \ip{ \a n + \b} \bmod m$. If $\bb a$ is asymptotically equal to an automatic sequence then $\a \in \QQ$.
\end{corollary}
}

\subsection{Multiplicative sequences}

A complex-valued sequence $\bb a = (a_n)_{n=1}^\infty$ is \emph{multiplicative} if $a_{nm} = a_n a_m$ for all $n,m \in \NN$ with $\gcd(n,m) = 1$, and \emph{completely multiplicative} if the requirement that $\gcd(n,m) = 1$ can be dropped. The problem of classifying automatic multiplicative sequences\footnote{
A slight notational inconvenience stems from the fact that automatic sequences are most naturally defined on $\NN_0$ and multiplicative sequences --- on $\NN$. In order to avoid this issue we adopt the convention where each multiplicative sequence $\bb a = (a_n)_{n=1}^\infty$ indexed by $\NN$ is extended to $\NN_0$ by setting $a_0 = 0$.
}
 has been considered in many papers, including \cite{ Yazdani-2001, SchlagePuchta-2003, BorweinCoons-2009, BorweinChoiCoons-2010, Coons-2010,  BellBruinCoons-2012, BellCoonsHare-2014, Allouche-2018,  KlurmanKurlberg-2019, KlurmanKurlberg-2020, Konieczny-2020, Li-2020}, culminating in a complete resolution in \cite{KoniecznyLemanczykMullner-2022}.
\begin{theorem}\label{thm:KLM}
	Let $\bb a$ be a finitely-valued multiplicative sequence. Then $\bb a$ is automatic if and only if it is $p$-automatic for a prime $p$ and takes the form
	\begin{align*}
		a_{p^in} &= c_i b_{n}
		& \text{for } i \in \NN_0,\ n \in \NN,\ p \nmid n, 
	\end{align*}
	where $\bb b$ and $\bb c$ are eventually periodic sequences.
\end{theorem} 
 
In this section we will show that each asymptotically automatic completely multiplicative sequence behaves in a periodic manner on all prime powers, except possibly for some more complicated behaviour on small primes. It it very likely that one can obtain a complete classification, in the spirit of \cite{KoniecznyLemanczykMullner-2022}, but for the sake of exposition we prove a somewhat weaker result.

\begin{proposition}\label{prop:class:multiplicative}
	Let $\bb a$ be a sequence that is both multiplicative and asymptotically $k$-automatic. Then $\bb a \simeq 0$ or there exists a Dirichlet character $\chi$ such that $a_{p^i} = \chi(p^i)$ for each sufficiently large prime $p$ and each $i \in \NN$.
\end{proposition}

Before we approach the proof, we discuss some corollaries and preliminaries. First, since any modification of an automatic sequence on a set with density zero is asymptotically automatic, we have the following consequence, expressed purely in terms of automatic and multiplicative sequences.

\begin{corollary}
	Let $\bb a$ be a multiplicative sequence that is asymptotically equal to an automatic sequence. Then $\bb a \simeq 0$ or there exists a Dirichlet character $\chi$ such that $a_{p^i} = \chi(p^i)$ for each sufficiently large prime $p$ and each $i \in \NN$.
\end{corollary}

In particular, we see that some of the classical multiplicative sequences are not asymptotically equal to an automatic sequence.
\begin{corollary}
	There is no automatic sequence $\bb a$ with $\bb a \simeq \bm \mu$, where $\bm \mu$ denotes the M\"obius sequence given by
\[
	\mu_n = 
	\begin{cases}
		(-1)^k &\text{if } n = p_1p_2\dots p_k \text{ is the product of $k$ distinct primes,}\\
		0 &\text{if $n$ is divisible by a square.}
	\end{cases}
\]	
	 Likewise, there is no automatic sequence $\bb a$ with $\bb a \simeq \bm \lambda$, where $\bm \lambda$ denotes the Liouville sequence given by
	 \[
	 	\lambda_n = (-1)^k \text{ if } n = p_1p_2\dots p_k \text{ is the product of $k$ primes.}
	 \] 
\end{corollary}

The fact that there exist asymptotically shift-invariant sequences that are not asymptotically equal to a periodic sequence leads to some technical difficulties. In order to overcome them, we will need the following result, which is essentially contained in \cite{Klurman-2017}. We call a sequence $1$-bounded if it takes values in $\set{z \in \CC}{ \abs{z} \leq 1}$. To avoid breaking the flow of discussion, we delegate the proof to the appendix.

\begin{theorem}\label{thm:class:mult-periodic} 
	Let $\bb a$ be a $1$-bounded, finitely-valued multiplicative sequence that is asymptotically invariant under a shift. Then $\bb a \simeq \bf 0$ or $\bb a$ is periodic.
\end{theorem}

We will also need the following basic lemma.

\begin{lemma}\label{lem:class:0-along-AP}
	Let $\bb a$ be a multiplicative sequence. Suppose that there exists $Q,r \in \NN$ with $\gcd(r,Q) = 0$ such that $a_n = 0$ for almost all $n \equiv r \bmod Q$. Then $\bb a \simeq 0$. 
\end{lemma}
\begin{proof}
	Let $\cZ$ be the set of all primes $p$ such that $a_p = 0$. There are two cases to consider, depending on the size of $\cZ$.	
	Suppose first that $\sum_{p \in \cZ} 1/p = \infty$. Since $a_{n} = 0$ for each $n$ with $p \mid n$ and $p^2 \nmid n$ for some $p \in \cZ$, for each finite set $\cF \subset \cZ$ we have
\begin{align*}
	\bar{d}\bra{ \set{n}{a_n \neq 0} } \leq
	\prod_{p \in \cF} \bra{ 1- \frac{p-1}{p^2} } 
		\ll \exp\braBig{-\sum_{p \in \cF } \frac{1}{p} }.
\end{align*}	
	Since $\cF \subset \cZ$ was arbitrary, we conclude that $\bb a \simeq \bb 0$, as needed.
	
	Suppose next that $\sum_{p \in \cZ} 1/p < \infty$ and hence (bearing in mind the prime number theorem in arithmetic progressions) in each residue class $s + Q \ZZ$ with $\gcd(s,Q) = 1$ we can find infinitely many primes $p$ with $a_p \neq 0$. Since for each $n \in \NN_0$ with $a_n \neq 0$ that is not divisible by $p$ we have $a_{pn} = a_{p} a_n \neq 0$, it follows that
\begin{align*}
	0 &= {d}\bra{ 
	\set{n \in \NN_0}{
	\begin{aligned}
	& a_n \neq 0,
	\\	& 
	n \equiv r \bmod Q
	\end{aligned}
	}}
\geq
	\frac{1}{p}\bra{ \bar{d}\bra{
		\set{n \in \NN_0}
		{
	\begin{aligned}
	& a_n \neq 0,\\
	& n \equiv s^{-1}r \bmod Q
	\end{aligned}}
	} - \frac{1}{p}}. 
\end{align*}	
	Letting $s$ vary and passing to the limit $p \to \infty$, we conclude that $a_n = 0$ for almost all $n \in \NN_0$ coprime to $Q$, and hence also for almost all $n \in \NN_0$.	
\end{proof}

Before moving on to the proof of Proposition \ref{prop:class:multiplicative}, let us briefly discuss the strategy. We begin with a relatively quick reduction to the case where the sequence $\bb a$ is completely multiplicative. Next, using Proposition \ref{prop:basic:structure}, we find an automatic sequence $\phi$ which encodes much of the behaviour of $\bb a$.  (For technical reasons, we will work with a slightly different automatic sequence $\psi$ defined in the body of the argument, but the difference between the two is not relevant for now.) The benefit of working with $\phi$ as opposed to $\bb a$ is that $\phi$ is genuinely automatic and hence we are in position to apply existing results on automatic sequences to it. Our plan is to exploit multiplicative behaviour of $\phi$. Of course, $\phi$ takes values in some set $\Sigma_d$ ($d \in \NN$) which does not have multiplicative structure\footnote{The fact that $\Sigma_d$ is a subset of the multiplicative semigroup $(\NN_0,\cdot)$ is a red herring: $\Sigma_d$ is only used to index the elements of $\tilde \cN_k(\bb a)$ and any other $d$-element set would do equally well.} 
so in order to talk about multiplicative behaviour of $\phi$ we need to do more work. The idea is, essentially, to define multiplication on $\Sigma_d$ in a way that is compatible with $\phi$. For technical reasons, we are marginally more careful: We pick a multiplicatively rich integer $M$, a set $\cG \subset \Sigma_d$ and define the map $f \colon \NN_0 \to \cG_0 := \cG \cup \{0_{\cG}\}$ by setting $f(n) = \phi(n)$ if $n$ is coprime to $M$ and $f(n) = 0_{\cG}$ otherwise. (Formally, $\cG$ is simply the set of all values of $\phi$ on integers coprime to $M$; the symbol $\cG$ is chosen to emphasise that think of $\cG$ as a group.) We define the operation $\odot$ on $\cG$ by setting $f(n) \odot f(m) := f(nm)$ for all $n,m \in \NN_0$. Verification that this definition is well-posed and makes $\cG$ into a group occupies the bulk of the argument. Once this is accomplished, we can use the classification of multiplicative automatic sequences as a black box to conclude that $f$ is periodic. Letting $L$ denote a period of $f$, it is relatively easy to show that $\bb a$ is asymptotically invariant under the shift by $L$. Combined with the results discussed earlier, this finishes the argument.

\begin{proof}[Proof of Proposition \ref{prop:class:multiplicative}]
{
If $\bb a \simeq \bf 0$ there is nothing to prove, so assume that this is not the case. As a preliminary simplification, we will use an argument similar to one employed in \cite{Konieczny-2020} to reduce to the case where $\bb a$ is completely multiplicative. Let $p$ be a large prime. It follows from the pigeon-hole principle that there exists $i,r,r' \in \NN_0$ with $0 \leq r,r' < k^i/p$ and $\gcd(k,r) = \gcd(k,r') = 1$ such that $r \not \equiv r' \bmod p$ and 
\begin{align}\label{eq:class:858:51}
\bra{ a_{k^i n + r} }_{n=0}^\infty &\simeq \bra{ a_{k^i n + r'} }_{n=0}^\infty 
& \text{and} && \bra{ a_{k^i n + pr} }_{n=0}^\infty &\simeq \bra{ a_{k^i n + pr'} }_{n=0}^\infty. 
\end{align} 
As a consequence, for almost all $n \in \NN_0$ with $k^i n + r' \not \equiv 0 \bmod{p}$ we have
\begin{align}\label{eq:class:858:52}
	a_{k^i pn + pr} = a_{k^i pn + pr'} = a_p a_{k^i n + r'} = a_p a_{k^i n + r}.  
\end{align}
Let $j \in \NN$. By \eqref{eq:class:858:52}, for almost all $n \in \NN_0$ with $k^i n + r \equiv p^j \bmod{p^{j+1}}$ we have
\begin{align}\label{eq:class:858:53}
	a_{p^{j+1}} a_{\bra{k^i n + r}/p^j} = a_p a_{p^{j}} a_{\bra{k^i n + r}/p^j}.  
\end{align}
Using Lemma \ref{lem:class:0-along-AP} we conclude that $a_{p^{j+1}} = a_p a_{p^{j}}$. Thus, applying induction with respect to $j$ we conclude that $a_{p^j} = a_p^j$. Pick a threshold $p_0$ such that the equality $a_{p^j} = a_p^j$ holds for all primes $p \geq p_0$ and all $j \in \NN$, and consider the sequence $\bar{\bb a}$ given by
\begin{align}\label{eq:class:858:54}
	\bar a_n =
	\begin{cases}
		0 & \text{if $n$ is divisible by a prime $p < p_0$},\\
		a_n & \text{otherwise}. 
	\end{cases}
\end{align}
Then $\bar{\bb a}$ is completely multiplicative, asymptotically $k$-automatic, and takes the same value as $\bb a$ on all powers of sufficiently large primes. Hence, replacing $\bb a$ with $\bar{\bb a}$ if necessary, we may assume that $\bb a$ is completely multiplicative.

	Pick $d$, $\bb a^{(0)}, \bb a^{(1)},\dots, \bb a^{(d-1)}$ and $\phi \colon \Sigma_k^* \to \Sigma_{d}$ as in Proposition \ref{prop:basic:structure}.
Let $H$ be an integer that is sufficiently multiplicatively rich that $\phi(0^H u) = \phi(0^{2H} u)$ for all $u \in \Sigma_k^*$, and put $K := k^H$. Define a map $\psi \colon \NN_0 \to \Sigma_d$ by $\psi(m) = \phi(0^h (m)_k)$, where $h$ is the unique integer	satisfying $H \leq h < 2H$ and $h+\abs{(m)_k} \equiv 0 \bmod H$. This construction ensures that for each $m \in \NN_0$ and for each $i \in \NN_0$ that is sufficiently large in terms of $m$ (more precisely: $K^{i-1} > m$) we have
\begin{equation}\label{eq:class:858:1}
	\brabig{ a_{K^i n + m} }_{n=0}^\infty \simeq 
	\brabig{ a_{n}^{(\psi(m))} }_{n=0}^\infty.
\end{equation}
We note that automaticity of $\phi$ implies that $\psi$ is $k$-automatic.
}

\begin{claim}\label{claim:class:div}  
	For each $q \in \NN$ and $m,m' \in \NN_0$, if $a_q \neq 0$ and $\psi(qm) = \psi(qm')$ then $\psi(m) = \psi(m')$.
\end{claim}
\begin{claimproof}
	For each sufficiently large $i$, for almost all $n \in \NN_0$ we have
	\begin{align*}
	a^{(\psi(m))}_n 
	&= a_{K^i n + m} 
	= a_q^{-1} a_{K^i q n + q m} 
	= a_q^{-1} a_{qn}^{(\psi(q m))} 
	\\&= a_q^{-1} a_{qn}^{(\psi(q m'))}  
	= a_q^{-1} a_{K^i q n + q m'}
	= a_{K^i n + m'}
	= a^{(\psi(m'))}_n.   
	\end{align*}
	It follows that $\bb a^{(\psi(m))} \simeq \bb a^{(\psi(m'))}$ and hence $\psi(m) = \psi(m')$. 
\end{claimproof}
	
{ 
	We will need the following standard technical result, which is closely related to \cite[Lem.\ 3.4]{Konieczny-2020} and is fully proved e.g.\ in \cite{KlurmanKonieczny-upcoming}. For the rest of the argument, we let $M$ denote the integer introduced below. We assume that $M$ is divisible by $k$ (otherwise we can freely replace $M$ with $kM$).
}
\begin{claim}\label{claim:class:equidist}  
	There exists a positive integer $M$ such that for each $j \in \Sigma_d$ with
	\[
		\dlog\bra{\set{n \in \NN_0}{\gcd(n,M) = 1,\ \psi(n) = j}} > 0,
	\]		
	 for each $q \in \NN$ comprime to $M$ and each $r \in \ZZ/q\ZZ$ we also have 
	\[
		\dlog\bra{\set{n \in \NN_0}{\gcd(n,M) = 1,\ \psi(n) = j,\ n \equiv r \bmod q}} > 0.
	\]
\end{claim}

\begin{claim}\label{claim:class:psi-1}  
	We have	
	\(
		\dlog\bra{\set{n \in \NN}{ \gcd(n,M) = 1,\ \psi(n) = \psi(1)}} > 0.
	\)	
\end{claim}
\begin{claimproof}  
	Pick any $j$ such that 
	\(
		\dlog\bra{\set{n \in \NN}{ \gcd(n,M) = 1,\ \psi(n) = j}} > 0.
	\)	
	Pick any $q$ belonging to that set, that is, $\psi(q) = j$ and $\gcd(q,M) = 1$. Consider the set 
	\[
		A = \set{n \in \NN}{ \gcd(n,M) = 1, \ \psi(q n) = j}.
	\]
	By Claim \ref{claim:class:equidist} we have $\dlog(A) > 0$, as 
	\[
		\dlog(qA) = 
		\dlog\bra{\set{n \in \NN}{\gcd(n,M) = 1, \ \psi(n) = j,\ n \equiv 0 \bmod q}} > 0.
	\]
	For each $n \in A$ we have $\psi(qn) = \psi(q) = j$. 
	If $a_q \neq 0$ then it follows from Claim \ref{claim:class:div} that $\psi(n) = \psi(1)$. Since $n \in A$ was arbitrary, we conclude that
	\[
	\dlog\bra{\set{n \in \NN}{ \gcd(n,M) = 1,\ \psi(n) = \psi(1)}} \geq \dlog(A) > 0,
	\]
	and the proof is complete. 
	It remains to consider the case where $a_q = 0$ for all $q$ as above. Inspecting the argument, we see that this is only possible if $a_{n} = 0$ for almost all $n \in \NN$ that are coprime to $M$, and hence (since $\bb a$ is multiplicative) $\bb a \simeq \bf 0$, contradicting an earlier assumption.
	\end{claimproof}
	
{ 
We will routinely use the following observation in conjunction with Claim \ref{claim:class:div} to remove the assumption that $a_q \neq 0$.
}
\begin{claim}\label{claim:class:non-zero}  
	For each $q \in \NN$ coprime to $M$ we have $a_q \neq 0$.
\end{claim}
\begin{claimproof}  
	For the sake of contradiction, suppose that $a_q = 0$. By Claims \ref{claim:class:equidist} and \ref{claim:class:psi-1}, there exists $r \in \NN$ such that $\psi(qr) = \psi(1)$. Let $i$ be a sufficiently large integer. Then
	\begin{align*}
	\bb 0 &= \bra{ a_{K^i q n + qr}  }_{n=0}^\infty \simeq \bra{ a_{K^i q n + 1}  }_{n=0}^\infty.
\end{align*}	 
Hence, it follows from Lemma \ref{lem:class:0-along-AP} that $\bb a \simeq 0$, contradicting an earlier assumption.
\end{claimproof}

\begin{claim}\label{claim:class:semigroup}  
	Let $q,r \in \NN$ be coprime to $M$. If $\psi(q) = \psi(r) = \psi(1)$ then $\psi(qr) = \psi(1)$.
\end{claim}
\begin{claimproof}  
	By Claim \ref{claim:class:equidist}, there exists $u \in \NN$ coprime to $M$ and such that 
	\[
		\psi(qru) = \psi(1) = \psi(q).
	\]
	Hence, by Claim \ref{claim:class:div}, we have
	\[
		\psi(ru) = \psi(1) = \psi(r).
	\]
	Using Claim \ref{claim:class:div} again, we obtain
	\[
		\psi(u) = \psi(1) = \psi(qru).
	\]
	Yet another application of Claim \ref{claim:class:div} yields
	\[
		\psi(1) = \psi(qr),
	\]
	as needed.
\end{claimproof}

\begin{claim}\label{claim:class:inverse}  
	Let $q,q',r \in \NN$ be coprime to $M$. If $\psi(q) = \psi(q')$ and $\psi(qr) = \psi(1)$ then $\psi(q'r) = \psi(1)$.
\end{claim}
\begin{claimproof}  
	By Claim \ref{claim:class:equidist}, there exists $u \in \NN$ coprime to $M$ such that
	\[
		\psi(q'r u) = \psi(1) = \psi(qr). 
	\]
	Hence, by repeated applications of Claim \ref{claim:class:div} we obtain
	\begin{align*}
		\psi(q' u) &= \psi(q) = \psi(q') \\
		\psi(u) &= \psi(1) = \psi(q'r u) \\
		\psi(1) &= \psi(q' r),
	\end{align*}
	as needed.
\end{claimproof}

\begin{claim}\label{claim:class:multi}  
	Let $q,q',r,r' \in \NN$ be coprime to $M$. If $\psi(q) = \psi(q')$ and $\psi(r) = \psi(r')$ then $\psi(qr) = \psi(q'r')$.
\end{claim}
\begin{claimproof}  
	By Claim \ref{claim:class:equidist}, there exist $\hat q, \hat r \in \NN$ coprime to $M$ and such that 
	\[ \psi(q \hat q) = \psi(r \hat r) = \psi(1). \]	
	By Claim \ref{claim:class:inverse} we also have
	\[ \psi(q' \hat q) = \psi(r'\hat r) = \psi(1). \]	
	By Claim \ref{claim:class:semigroup}, it follows that
	\[ \psi( q r \hat q \hat r) = \psi(1) = \psi(q' r' \hat q \hat r).\]
	Thus, by Claim \ref{claim:class:div} we conclude that
	\( \psi( q r ) = \psi(q' r' ),\) 
	as needed.
\end{claimproof}

{ 
We are now ready to use $\psi$ to construct a multiplicative sequence that is genuinely automatic (as opposed to merely asymptotically automatic). Define 
\[
	\cG := \set{ \psi(q) }{ q \in \NN,\ \gcd(q,M) = 1},
\]
and $\cG_0 = \cG \cup \{0_{\cG}\}$, where $0_{\cG} \not \in \cG$.  Let $f \colon \NN_0 \to \cG$ be given by 
\[
	f(n) = 
	\begin{cases}
		\psi(n) & \text{if } \gcd(n,M) = 1;\\
		0_{\cG} & \text{if } \gcd(n,M) > 1.
	\end{cases}
\]
We define a binary operation $\odot$ on $\cG$ by 
\[ \psi(q) \odot \psi(r) = \psi(qr)\]
for $q,r \in \NN$ coprime to $M$; this definition is well-posed thanks to Claim \ref{claim:class:multi}. We extend $\odot$ to $\cG_0$ by setting $g \odot 0_{\cG} = 0_{\cG} \odot g = 0_{\cG}$ for all $g \in \cG_0$. It is routine to verify for all $n,m \in \NN_0$ we have
\[
	f(nm) = f(n) \odot f(m).
\]
It follows that $(\cG_0,\odot)$ is a finite abelian semigroup, which is the image of $(\NN_0,\cdot)$ under the semigroup homomorphism $f$. 
Moreover, $\psi(1)$ is the identity element in $\cG$, and by Claim \ref{claim:class:equidist}, each element of $\cG$ has an inverse. Hence, $(\cG,\odot)$ is a group.

}	
	
\begin{claim}\label{claim:class:f-period}  
	The map $f$ is periodic.
\end{claim}
\begin{claimproof}  
	If $(\cG,\odot)$ is isomorphic to the cyclic group $C_n$ for some $n \in \NN$ then $\cG_0$ can be identified with the subset of complex numbers $\set{z \in \CC}{z^n = 1 \text{ or } z = 0}$, under which identification $f$ can be construed as a multiplicative map $f \colon \NN_0 \to \CC$. 	Hence, the conclusion follows from Theorem \ref{thm:KLM}. The general case can readily be derived by considering characters of $(\cG,\odot)$.\end{claimproof}	

{ 
Let $L$ be a period of $f$ that is also a multiple of $M$. Consider the complex-valued multiplicative sequence $\bb a'$, given by
\[
	a'_n := 
	\begin{cases}
		a_n & \text{if } \gcd(n,k) = 1;\\
		0 & \text{if } \gcd(n,k) > 1.
	\end{cases}
\]
}

\begin{claim}\label{claim:class:a-asympt-period}  
	The sequence $\bb a'$ is asymptotically invariant under shift by $L$.
\end{claim}
\begin{claimproof}  
	For each $n$ with $\gcd(k,n) > 1$ we have $a'_n = a'_{n+L} = 0$, so we only need to verify that $a'_n = a'_{n+L}$ for almost all $n$ coprime to $k$.
	Let $P_{i,r}$ be the arithmetic progression given by $P_{i,r} = K^i \NN_0 + r$, where $i,r \in \NN_0$, $r < K^{i-1} - L$, and $\gcd(r,M) = 1$. For almost all $n \in \NN_0$ we have
	\[
		a_{K^i n + r}' = a_{K^i n + r} = a^{(\psi(r))}_{n} = a^{(\psi(r+L))}_{n} = a_{K^i n + r+L} = a_{K^i n + r+L}'. 
	\]
	Hence, for almost all $n \in P_{i,r}$ we have $a_{n+L}' = a_{n}'$. Thus, letting 
	\[ 
		B_j := \set{n \in \NN}{\gcd(n,k) = 1} \setminus \textstyle \bigcup_{i = 1}^{j} \bigcup_{r} P_{i,r},
	\]
	where the union runs over all integers $r$ with $0 \leq r < K^{i-1} - L$ and $\gcd(r,M) = 1$, it will suffice to show that $d(B_j) \to 0$ as $j \to \infty$. It is not hard to construct a word $w \in \Sigma_K^*$ that is not contained in the last $j$ digits of the  base-$K$ expansion of any $n \in B_j$. (For instance, one can pick  $w = (0^{h} 1)^M$ where $h$ is a sufficiently large and multiplicatively rich integer; the key idea is that if $w$ appears in $(n)_K$ and $\gcd(n,k) = 1$ then $n$ belongs to $M$ different arithmetic progression of the form $K^{i} \NN_0 + r$ with $r < K^{i-h} \leq K^{i-1} - L$ and at least one of those progressions satisfies $\gcd(r,M) = 1$.) It follows that
	\[
		d(B_j) \leq K^{-j} \cdot \#\set{n \in [K^j]}{(n)_K \text{ does not contain $w$}} \to 0 \text{ as } j \to \infty,
	\]     
	which together with earlier considerations proves the claim.
\end{claimproof}

{ 
Combining Claim \ref{claim:class:a-asympt-period} with Theorem \ref{thm:class:mult-periodic} finishes the argument.
}
\end{proof}

\color{checked}
\section{Regular sequences}\label{sec:regular}

The fact that automatic sequences are necessarily finitely-valued poses a significant limitation. This motivated Allouche and Shallit to introduce a closely related notion of a \emph{regular sequence} \cite{AlloucheShallit-1992}, \cite{AlloucheShallit-2003}. Let $R$ be an integral domain (e.g.\ $R = \ZZ$ or $R = \FF_p$) and let $K$ be the field of fractions of $R$ (e.g.\ $K = \QQ$ or $K = \FF_p$). A sequence $\bb a \in R^\infty$ is $k$-regular if and only if the  $R$-module  generated by $\cN_k(\bb a)$
\[
	\cM_k(\bb a) := \linspan_{R}  \cN_k(\bb a)
\]
is finitely generated. In the case where $R = \ZZ$, this is equivalent to the assertion that $\cN_k(\bb a)$ spans a finite-dimensional space over $\QQ$. In particular, $k$-regular sequences include $k$-automatic sequences and polynomial sequences. Conversely, one can show that each finitely-valued $k$-regular sequence is $k$-automatic.

In analogy with asymptotically $k$-automatic sequences, we can define asymptotically $k$-regular sequences. Thus, a sequence $\bb a \in R^\infty$ is asymptotically $k$-regular if and only if $\tilde \cM_k(\bb a) := \cM_k(\bb a)/{\simeq}$ is a finitely-generated $R$-module
or, equivalently, if there exists a finite number of sequences $\bb a^{(0)},\bb a^{(1)}, \dots, \bb a^{(d-1)}$ such that for each $\bb b \in \cM_k(\bb a)$ there exist $t_0,t_1,\dots,t_{d-1} \in R$ such that \[
	\bb b \simeq t_0 \bb a^{(0)} + t_1 \bb a^{(1)} + \dots + t_{d-1} \bb a^{(d-1)}.
\]
Clearly, asymptotically $k$-regular sequence include $k$-regular sequences and asymptotically $k$-automatic sequences. It also remains true, with essentially the same argument, that each finitely-valued asymptotically $k$-regular sequence is asymptotically $k$-automatic.

An important feature of real-valued $k$-regular sequences is that their rate of growth is at most polynomial, meaning that if $\bb a \in \RR^\infty$ is $k$-regular then 
\[\limsup_{n \to \infty} \frac{\log \abs{a_n}}{\log n} < \infty.\]
This is no longer true for asymptotically $k$-regular sequences, which, as we will see below, can grow arbitrarily fast. This suggests that asymptotically $k$-regular sequences exhibit significantly less structured behaviour than $k$-regular sequences.

\begin{proposition}
	For each $f \colon \NN_0 \to \NN_0$, there exists a $k$-regular sequence $\bb a \in \ZZ^\infty$ such that $a_n \geq f(n)$ for all $n \in \NN_0$. 
\end{proposition}
\begin{proof}
	We rely (once more) on the construction in \cite[Sec.\ 4]{JK-Cobham-asympt}, cf.\ Prop.\ \ref{prop:freq:not-converge}. Let 
	\( \cH = \{H_i\}_{i=0}^\infty = \{2^{\alpha_i} 3^{\beta_i} \}_{i=0}^\infty\)
	 be as in Proposition \ref{prop:freq:not-converge}, and let $\bb a$ be given by
\begin{equation}\label{eq:49:def-f}
	a_n = F(\b_i) \text{ for } n \in [H_i,H_{i+1}) \text{ and } i \in \NN_0,
\end{equation}
where $F \colon \NN_0 \to \NN_0$ is an increasing function that remains to be specified. Following the same argument as in \cite[Lem.\ 4.2]{JK-Cobham-asympt}, we see that $\#\cN_2(\bb a) = 1$, which in particular implies that $\bb a$ is asymptotically $2$-regular. 

Since for each $\gamma \in \NN$, the union of all intervals $[H_i,H_{i+1})$ such that $\beta_i \leq \gamma$ has density $0$, we can construct a non-decreasing sequence $\bm \gamma \in \NN_0^\infty$ such that $\gamma_n \to \infty$ as $n \to \infty$ and for almost all $n$, letting $i$ denote the unique index such that $n \in [H_i,H_{i+1})$, we have $\gamma_n \leq \beta_i$. Thus, $a_n \geq F(\gamma_n)$ for almost all $n$. Picking $F(\gamma) := \max\set{ f(n) }{ \gamma_n \leq \gamma}$ we conclude that $a_n \geq f(n)$ for almost all $n$.

Finally, let $\bb a'$ be given by $a'_n = \max( a_n, f(n))$. Then directly from definition we have $a'_n \geq f(n)$ for all $n \in \NN_0$ and $\bb a'$ is asymptotically $2$-regular because $\bb a' \simeq \bb a$.
\end{proof}

In \cite{Bell-2006}, Bell proved the analogue of Cobham's theorem for $k$-regular sequences. With a similar argument, we can obtain a slightly stronger version, which is analogous to \cite[Thm.\ B]{JK-Cobham-asympt}. We will say that a sequence $\bb a \in R^\infty$ satisfies a linear recurrence (over $R$) if there exists $d \in \NN$ and $t_1,t_2,\dots,t_d \in R$ such that 	
	\begin{align*}
		a_{n+d} &= t_1 a_{n+d-1} + t_2 a_{n+d-2} + \dots + t_d a_n 
		&\text{for all $n \in \NN_0$.}
	\end{align*}
If $R = \ZZ$ and $\bb a$ is additionally $k$-regular for some $k \geq 2$ then this is equivalent to the requirement that there exists $Q \in \NN$ and polynomial sequences $p_m$ for $0 \leq m < Q$ such that $a_n = p_m(n)$ for each sufficiently large $n \equiv m \bmod Q$. For the sake of simplicity, we only consider the case where $R = \ZZ$ and the bases $k,\ell$ are coprime.

\begin{theorem}\label{thm:reg:main}
	Let $k,\ell \in \NN$ be coprime, and let $\bb a \in \ZZ^\infty$ be a sequence that is both $k$-regular and asymptotically $\ell$-regular. Then $\bb a \simeq \bb a'$ for a $k$-regular sequence $\bb a' \in \ZZ^\infty$ that satisfies a linear recurrence.
\end{theorem}

\renewcommand{\Lambda}{T}

Before we approach the proof, it will be convenient to introduce some preliminaries. For $u \in \Sigma_k$ we define a linear map $\Lambda_u \colon \QQ^\infty \to \QQ^\infty$, given by
\begin{align}\label{eq:reg:def-of-T}
	\Lambda_u(\bb a) &:= \bra{a_{k^{\abs{u}}n + [u]_k}}_{n=0}^\infty, & \bb a \in \QQ^\infty.
\end{align}
The operators $\Lambda_u$ follow a familiar composition rule: $\Lambda_{uv} = \Lambda_v \Lambda_u$ for $u,v \in \Sigma_k^*$. 
We note that, for a $k$-regular sequence $\bb a$, the module $\cM_k(\bb a)$ and the space $\linspan_{\QQ} \cN_k(\bb a)$ are invariant under $\Lambda_{u}$. Since the space of sequences that are almost everywhere $0$ is closed under $\Lambda_u$, for each asymptotically $k$-regular $\bb a$ we also have the map $\Lambda_u \colon \tilde \cM_k(\bb a) \to \tilde \cM_k(\bb a)$. The same formula \eqref{eq:reg:def-of-T} defines a map $T_u \colon \FF_p^\infty \to \FF_p^\infty$ for prime $p$.

\begin{lemma}\label{lem:reg:invert}
	Let $k \in \NN$, let $p$ be a prime and let $\bb a, \bb a' \in \FF_p^\infty$ be $k$-regular sequences such that for all $i \in \Sigma_k$, the maps $\Lambda_i \colon \cM_k(\bb a) \to \cM_k(\bb a)$ and $\Lambda_i \colon \cM_k(\bb a') \to \cM_k(\bb a')$ are invertible. If $\bb a \simeq \bb a'$ then $\bb a = \bb a'$. 
\end{lemma}
\begin{proof}
	The sequence $\bb b \in \FF_p^\infty$ given by
	\[
		b_n = 
		\begin{cases}
			0 &\text{if } a_n = a_n',\\
			1 &\text{if } a_n \neq a_n',\\			
		\end{cases}
	\]	
	is almost everywhere zero and $k$-regular, and hence $k$-automatic, since $\FF_p$ is finite. It follows that there exists a word $w \in \Sigma_k^*$ such that $a_n = a_n'$ for all $n \in \NN_0$ such that $(n)_k$ contains $w$ as a subword. In particular, we have $\Lambda_w(\bb a) = \Lambda_w(\bb a')$. Since $\Lambda_w$ is invertible (on either $\cM_k(\bb a)$ or $\cM_k(\bb a')$) we conclude that $\bb a = \bb a'$.
\end{proof}

\begin{lemma}\label{lem:reg:sep-0}
	Let $k \in \NN$ and let $\bb a \in \QQ^\infty$ be a $k$-regular sequence. Then there exists $\theta > 0$ such that for each $\bb a' \in \cM_k(\bb a)$ either $\bb a' \simeq \bb 0$ or $\bar{d}(\set{n \in \NN_0}{a'_n \neq 0}) \geq \theta$.
\end{lemma}
\begin{proof}
Suppose, conversely, that we have a family $\bb a^{(i)} \in \cM_k(\bb a)$ ($i \in \NN_0$) with 
\begin{align*}
0 < \bar{d} \bra{\set{n \in \NN_0}{a^{(i)}_n \neq 0}} &\to 0 & \text{as $i \to \infty$}.
\end{align*}
Passing to a subsequence we can assume without loss of generality that 
\begin{align*}
\bar{d} \bra{\set{n \in \NN_0}{a^{(i)}_n \neq 0}} &\geq
2 \bar{d} \bra{\set{n \in \NN_0}{a^{(i+1)}_n \neq 0}} 
& \text{for each $i \in \NN_0$.}
\end{align*}
As a consequence, for each linear combination
\begin{align*}
	\bb b = \sum_{i = I}^{I+J} t_i \bb a^{(i)}
\end{align*}
with $I,J \in \NN_0$, $t_i \in \QQ$ and $t_I \neq 0$ we have 
\begin{align*}
\bar{d} \bra{\set{n \in \NN_0}{b_n \neq 0}} & \geq \bar{d} \bra{\set{n \in \NN_0}{a^{(i)}_n \neq 0}} \cdot \bra{1 - \frac{1}{2} - \dots - \frac{1}{2^J}} > 0.
\end{align*}
As a consequence, the sequences $\bb a^{(i)}$ are linearly independent, contradicting the assumption of regularity.
\end{proof}

\begin{proof}[Proof of Theorem \ref{thm:reg:main}]
	We first consider the case where for each $i \in \Sigma_k^*$ the map $\Lambda_i$ is invertible on $\linspan_{\QQ} \cN_k(\bb a)$. For each prime $p$ the sequence $\bb a \bmod p := \bra{ a_n \bmod p}_{n=0}^\infty$ is both $k$-automatic and asymptotically $\ell$-automatic. Hence, it follows from \cite[Thm.\ B]{JK-Cobham-asympt} that $\bb a \bmod p$ is asymptotically periodic. With the same argument as in \cite[Prop.\ 3.10]{Bell-2006}, we see that $\Lambda_u (\bb a) \bmod p$ is asymptotically periodic with period coprime to $k$ for each $u \in \Sigma_k^*$ with $\abs{u} \geq \dim \linspan_{\QQ} \cN_k(\bb a)$. Note that the maps $\Lambda_i \mod p \colon \cM_k(\bb a \bmod p) \to \cM_k(\bb a \bmod p)$ are invertible for all sufficiently large primes $p$. Thus, it follows from Lemma \ref{lem:reg:invert} that $\Lambda_u(\bb a) \bmod p$ is (exactly) periodic. Hence, we conclude from \cite[Thm.\ 3.4]{Bell-2006} that $\Lambda_u(\bb a)$ satisfies a linear recurrence. Applying this result to all $u \in \Sigma_k^*$ with $\abs{u} = \dim  \linspan_{\QQ} \cN_k(\bb a)$, we conclude that $\bb a$ satisfies a linear recurrence (c.f.\ \cite[Thm.\ 3.11]{Bell-2006}). Additionally, inspecting the proof of \cite[Thm.\ 3.4]{Bell-2006} we see that the length of said linear recurrence is bounded by a quantity dependent only on $\dim  \linspan_{\QQ} \cN_k(\bb a)$.

	Consider now the general case. Let $w \in \Sigma_k^*$ be such that $\dim \Lambda_{w} \bra{\linspan_{\QQ} \cN_k(\bb a)}$ is least possible. Then for each $u \in \Sigma_k^*$ such that $w$ is a subword of $u$, the sequence $\bb a' = \Lambda_{u}(\bb a)$ has the property that for each $i \in \Sigma_k$, the restriction of $\Lambda_i$ to $\cM_k(\bb a')$ is invertible. Hence, we conclude from the special case considered above that the restriction of $\bb a'$ to each residue class modulo some $Q = O(1)$ coincides with a polynomial sequence of degree at most $d = O(1)$. Thus, we have constructed a large family of arithmetic progressions of the form $k^iQ \NN_0 + r$ ($i,r \in \NN_0$) where $\bb a$ behaves in the expected way; our remaining task is to understand how the restrictions of $\bb a$ to different progressions are interrelated.
	
	 For $i \in \NN_0$, let $\bb b^{(i)}$ denote the sequence given by
\begin{align}\label{eq:reg:27:01}
	b^{(i)}_n &= \sum_{m=0}^{d+1} (-1)^m \binom{d+1}{m} a_{n + k^i Q m} & \text{for } n \in \NN_0.	
\end{align}
Then $b^{(i)}_n = 0$ for each $n \in \NN_0$ such that $(n)_k^i$ contains $w$ as a subword. For $j,s \in \NN_0$ such that $s + k^i Q (d+1) < \ell^j$, consider further the sequence $\bb c^{(i,j,s)}$ given by
\begin{align}\label{eq:reg:27:02}
	c^{(i,j,s)}_n &= b^{(i)}_{\ell^j n + s} & \text{for } n \in \NN_0.	
\end{align}
Inspecting the definitions, we see that $\bb c^{(i,j,s)} \in \cM_\ell(\bb a)$. Moreover, since $k$ and $\ell$ are coprime, we have
\begin{align*}&\bar{d} \bra{\set{n \in \NN_0}{c^{(i,j,s)}_n \neq 0}} 
	\\ \leq
 &	\bar{d} \bra{\set{n \in \NN_0}{w \text{ does not appear in } (n)_k^i}} \to 0 \text{ as } i \to \infty.
\end{align*}
Thus, it follows from Lemma \ref{lem:reg:sep-0} that there exists $i \in \NN_0$ such that for all $j,s$ like above we have $\bb c^{(i,j,s)} \simeq 0$. In other words, $b^{(i)}_{n} = 0$ for almost all $n \in \ell^j \NN_0 + s$. Note that almost all positive integers belong to $\ell^j \NN_0 + s$ for some admissible $j,s$; more precisely, we have
\begin{align*}
	\underline{d}\bra{ \bigcup_{j=0}^{J} \bigcup_{s = 0}^{\ell^j - k^i Q (d+1)} \bra{ \ell^j \NN_0 + s} 
	} &\to 1 & \text{as } J \to \infty.
\end{align*}

It follows that $\bb b^{(i)} \simeq 0$. Thus, we may partition $\NN_0$ into a union of pairwise disjoint arithmetic progressions  $(I_m)_{m=1}^\infty$ with step $k^i Q$ such that $\sum_{m =1}^{M} \abs{I_m}/M \to \infty$ as $M \to \infty$ and for each $m$, $(a_n)_{n \in I_m}$ coincides with a polynomial sequence of degree at most $d$. Consider any pair of infinite arithmetic progressions $P_1 = k^{i_1}Q\NN_0 + r_1$ and $P_2 = k^{i_2}Q\NN_0 + r_2$ with $0 \leq r_1 < k^{i_1}$, $0 \leq r_2 < k^{i_2}$ such that the restrictions of $\bb a$ to either progression coincides coincides with a polynomial of degree at most $d$, as constructed earlier. Suppose additionally that $i_1,i_2 \geq i$ and $r_1 \equiv r_2 \bmod k^i$. Then we can find $m$ such that $\# (P_1 \cap I_m), \# (P_2 \cap I_m) > d$ and as a consequence the restrictions of $\bb a$ to $P_1$ and $P_2$ (and $I_m$) are given by the same polynomial. Varying the choice of $P_1$ and $P_2$ we conclude that $\bb a$ coincides with a polynomial of degree at most $d$ almost everywhere on each residue class modulo $k^iQ$.
\end{proof}

We close this section with a conjecture, which appears to be the correct analogue of Cobham's theorem (cf.\ \cite[Thm.\ 1.4]{Bell-2006} and \cite[Thm.\ A]{JK-Cobham-asympt}).
\begin{conjecture}
	Let $k,\ell \in \NN$ be multiplcatively independent, and let $\bb a \in \ZZ^\infty$ be a sequence that is both asymptotically $k$-regular and asymptotically $\ell$-regular. Then $\bb a$ satisfies a linear recurrence almost everywhere, i.e., there exists $d \in \NN$ and $t_1,t_2,\dots,t_d \in \ZZ$ such that
	\begin{align*}
		a_{n+d} &= t_1 a_{n+d-1} + t_2 a_{n+d-2} + \dots + t_d a_n 
		&\text{for almost all $n \in \NN_0$.}
	\end{align*}
\end{conjecture} 
\appendix

\section{Shift invariant multiplicative sequences}\color{checked}

\begin{proof}[{Proof of Theorem \ref{thm:class:mult-periodic}}]
{\color{checked}
	Suppose first that $\bb a$ is invariant under the shift by $1$. Then in particular
	\begin{align}\label{eq:class:902:1}	
		\frac{1}{N} \sum_{n = 1}^{N} \abs{ a_{n+1} - a_n} = 0.
	\end{align}
	Hence, K\'{a}tai's conjecture proved by Klurman \cite[Thm.\ 1.8]{Klurman-2017} implies that either 
	\begin{align}\label{eq:class:902:2}
		\frac{1}{N} \sum_{n = 1}^{N} \abs{a_n} = 0,
	\end{align}
	or there exists $s \in \CC$ with $\Re(s) < 1$ such that $a_n = n^s$ for all $n \in \NN$. In the former case, since the set $\set{ \abs{a_n} }{ n \in \NN,\ a_n \neq 0}$ is bounded away from $0$, we have $\bb a \simeq \bf 0$. In the latter case, since $a_n$ is finitely-valued, we have $s = 0$, meaning that $\bb a = \bf 1$.
}

{
	Let us now consider the general case, where $\bb a$ is invariant under the shift by some $d \in \NN$, not necessarily equal to $1$. We will heavily rely on the material in \cite{Klurman-2017}. If $\bb a \simeq \bf 0$ then there is nothing left to prove, so assume that this is not the case. 
	It will be convenient to decompose $\bb a$ as the product of multiplicative sequences $\bb a^{(p)}$ corresponding to different primes $p$. If $p$ is a prime then the sequence $\bb a^{(p)}$ is specified for primes $r$ and exponents $i$ by
	\begin{align*}
		a^{(p)}_{r^i} =
		\begin{cases}
			a_{r^i} 	&\text{if } r = p,\\
			1 				&\text{otherwise},
		\end{cases}		
	\end{align*}
	and extended to $\NN$ by multiplicativity. Clearly, $a_n = \prod_{p} a_n^{(p)}$ for all $n \in \NN$. We will also use the sequences $\bb c$ and $\bb c^{(p)}$ given by $c_n = \absnormal{a_n}^2$ and $c_n^{(p)} = \absnormal{a_n^{(p)}}^2$.
}

Let us consider the limit
	\begin{align}\label{eq:class:718:1}
		A &:= \lim_{N \to \infty} \frac{1}{N} \sum_{n = 1}^{N} a_{n} \bar a_{n+d}.
	\end{align}
	On one hand, because $\bb a$ is asymptotically invariant under the shift by $d$, we have
	\begin{align}\label{eq:class:718:2}
		A &= \lim_{N \to \infty} \frac{1}{N} \sum_{n = 1}^{N} a_{n} \bar a_{n} = \lim_{N \to \infty} \frac{1}{N} \sum_{n = 1}^{N} c_{n}.
	\end{align}	
	By Hal\'{a}sz's theorem \cite{Halasz-1971} (cf.\ eq.\ (1) in \cite{Klurman-2017}), the mean value of $\bb c$ exists (thus, $A$ is well-defined) and is the product of contributions coming from the primes:
	\begin{align}\label{eq:class:718:3}
		A &= \prod_{p} C_p, & \text{where} && C_p & := \lim_{N \to \infty} \frac{1}{N} \sum_{n = 1}^{N} c_{n}^{(p)}.
	\end{align}		
Note that $A$ is real and $A > 0$ because the limit in \eqref{eq:class:718:2} converges and $\bb a$ is finitely-valued and not asymptotically equal to $\bf 0$.

The ``distance'' between $\bb a$ and another $1$-bounded multiplicative sequence $\bb b$ is defined by
	\[
		\D(\bb a,\bb b) := \sqrt{ \sum_{p \text{ prime}} \frac{1 - \Re\bra{a_p \bar b_p}}{p} } \in [0,\infty].
	\]

\begin{claim}
	If $\D(\bb a, \bf 1) < \infty$ then for each prime $p$ there exists a root of unity $\lambda$ such that $a_{n+d}^{(p)} = \lambda a_n^{(p)}$ for almost all $n \in \NN$.
\end{claim}
\begin{claimproof}
	By Theorem 1.3 in \cite{Klurman-2017}, the limit in \eqref{eq:class:718:1} admits a factorisation:
	\begin{align}\label{eq:class:718:4}
		A &= \prod_{p} B_p, & \text{where} && B_p & := \lim_{N \to \infty} \frac{1}{N} \sum_{n = 1}^{N} a_{n}^{(p)} \bar a_{n+d}^{(p)}.
	\end{align}	
	
	An application of the Cauchy--Schwarz inequality shows that $\abs{B_p} \leq C_p$, and hence comparing \eqref{eq:class:718:3} and \eqref{eq:class:718:4} we conclude that $\abs{B_p} = C_p$ for all primes $p$. Let $p$ be a prime and pick $\lambda \in \CC$ with $\abs{\lambda}$ such that $\lambda B_p \in \RR_+$. Then 
	\begin{align}\label{eq:class:718:5}
		0 \leq \lim_{N \to \infty} \frac{1}{N} \sum_{n = 1}^{N} \abs{a_{n+d}^{(p)} - a_n^{(p)}}^2 = 2 C_p - 2 \abs{B_p} = 0.
	\end{align}	
	Since $\bb a$ is finitely-valued, it follows  	
that $a_{n+d}^{(p)} = \lambda a_n^{(p)}$ for almost all $n$.	For each $z \in \CC$ that appears in $\bb a$ with positive upper frequency, the same is true of $\lambda z, \lambda^2 z,\dots$, and consequently $\lambda$ is a root of unity.
\end{claimproof}
	
\begin{claim}
	Let $p$ be a prime and let $\lambda$ be a root of unity such that $a_{n+d}^{(p)} = \lambda a_n^{(p)}$ for almost all $n \in \NN$. Then $\bb a^{(p)}$ is periodic. If $p \nmid 2d$ then $\bb a^{(p)} = \bf 1$.
\end{claim}
\begin{claimproof}	
Let $h := \nu_p(d)$ and for $i \in \NN_0$ let $z_i := a_{p^i}$. For each $n \in \NN$ with $\nu_p(n) =: i$ we have
	\begin{align}\label{eq:class:242:1}
		\nu_p(n+d) =
		\begin{cases}
			i &\text{if } i < h,\\
			h &\text{if } i > h,\\
			h + \nu_p\bra{\frac{n+d}{p^h}} & \text{if } i = h.
		\end{cases}
	\end{align}
	If $p \neq 2$ then $\nu_p( (n+d)/p^h)$ can take any value in $\NN_0$, and if $p=2$ then $\nu_p( (n+d)/p^h)$ can take any value in $\NN$. We also note that each possible pair of values $\nu_p(n), \nu_p(n+d)$ occurs for a periodic set of values of $n$, and hence with positive frequency. Thus, assuming for a moment that $p \neq 2$, it follows from \eqref{eq:class:242:1} that  
\begin{align}
\label{eq:class:242:2a}
	\lambda z_i &= z_i \text{ for all } i < h, \\
\label{eq:class:242:2b}
	\lambda z_i &= z_h \text{ for all } i > h, \\
\label{eq:class:242:2c}
	\lambda z_h &= z_i \text{ for all } i \geq h.
\end{align}
If $p = 2$ then \eqref{eq:class:242:2c} only holds for $i > h$.
	
	It follows from \eqref{eq:class:242:2b} that $(z_i)_{i=0}^\infty$ is eventually constant: $z_i = \lambda^{-1} z_{h}$ for all $i > h$. As a consequence, $\bb a^{(p)}$ is periodic with period $p^{h+1}$. This proves the first part of the claim. For the second part, taking $i = h = 0$ in \eqref{eq:class:242:2c} we conclude that $\lambda = 1$. Hence, \eqref{eq:class:242:2b} implies that $z_i = z_0 = 1$ for all $i \geq 1$ and consequently $\bb a^{(p)} = \bf 1$.
\end{claimproof}	
	
	In the case where $\D(\bb a,\bf 1) < \infty$, we are now ready to finish the argument. Indeed, in this case each of the sequences $\bb a^{(p)}$ is periodic, and each except for finitely many of them is the constant sequence $\bf 1$, and hence  $\bb a$ is periodic. 	In general, it is not necessarily the case that $\D(\bb a,\bf 1) < \infty$ but we have the following slightly weaker fact.
	 
\begin{claim}
	There exists $t \in \RR$ and a primitive Dirichelet character $\chi$ with some conductor $q$ such that for the sequence $\bb b$ given by $b_n = \chi(n)n^{it}$ we have $\D(\bb a,\bb b) < \infty$. 
\end{claim}
\begin{claimproof}
	Since the limit defining $A$ converges, we also have
	\[
		\lim_{N\to \infty} \frac{1}{\log N} \sum_{n=1}^N \frac{a_n \bar a_{n+d}}{n} = A > 0.
	\]
This puts us in the same position as in the second half of the proof of Lemma 4.3 in \cite{Klurman-2017}. The claim follows by combining Tao's result on two-point correlations of multiplicative functions \cite{Tao-2016}, (\cite[Thm.\ 4.1]{Klurman-2017}) with a technical lemma due to Elliott \cite{Elliott-2010} (\cite[Lem.\ 4.2]{Klurman-2017}) and repeating the reasoning at the end of the proof of Lemma 4.3 in \cite{Klurman-2017} verbatim.
\end{claimproof}	 

\begin{claim}
	We have $t = 0$.
\end{claim}
\begin{claimproof}
For the sake of contradiction, suppose that $t \neq 0$ and let $\tau := t/\pi$. Since $\bb a$ and $\chi$ are finitely-valued, there exist $\theta \in [0,1)$ and $\e,\delta > 0$ such that for each sufficiently large prime $p$ with $\abs{ \fp{ \tau \log p } - \theta } < \delta$ we have $\Re(a_p \bar \chi(p) p^{-it}) < 1 - \e$. It follows from the prime number theorem that for each sufficiently large $n$, the number of primes $p$ with $\abs{  \tau \log p - n - \theta } < \delta$ is
\[
	\frac{\exp\bra{\frac{n+\theta+\delta}{\tau}}}{(n+\theta+\delta)/\tau} 
	- \frac{\exp\bra{\frac{n+\theta-\delta}{\tau}}}{(n+\theta-\delta)/\tau} + o\bra{\frac{\exp(\frac{n}{\tau})}{n}}
	= (c+o(1))\frac{\exp(\frac{n}{\tau})}{n},
\]
where $c > 0$ is a constant dependent on $\delta$, $\tau$ and $\theta$. It follows that
\begin{align*}
	\D(\bb a, \bb b)^2 \geq \e \sum_{n = 1}^\infty
	\sum_{\substack{ p \text{ prime} \\ \abs{  \tau \log p - \theta -n} < \delta}} \frac{1}{p} 
	 \gg  \e \sum_{n=1}^\infty \frac{c+o(1)}{n} = \infty,
\end{align*}
contrary to the defining property of $\bb b$.
\end{claimproof}	 
	 
\begin{claim}\label{claim:class:mult-per:5}
	The sequence $\bb a \bar{\bb b} = \brabig{a_n \bar \chi(n)}_{n=1}^\infty$ is periodic.
\end{claim}	 
\begin{claimproof}
	The sequence $\bb a \bar{\bb b}$ is multiplicative, asymptotically invariant under the shift by $qd$, and satisfies $\D(\bb a \bar{\bb b}, \bf 1) = \D(\bb a, \bb b) < \infty$. Thus, applying the earlier reasoning with $\bb a \bar{\bb b}$ in place of $\bb a$ we conclude that $\bb a \bar{\bb b}$ is periodic.
\end{claimproof}
	 
\begin{claim}
	The sequence $\bb a$ is periodic.
\end{claim}
\begin{claimproof}
	It follows from Claim \ref{claim:class:mult-per:5} that the restriction of $\bb a$ to integers coprime to $qd$ is periodic with period $qd$. For integers $m,m',r \in \NN_0$ such that $m,m'$ divide a power of $qd$ and $0 \leq r < qd \lcm(m,m')$, consider the set $P$ of integers $n$ such that
\begin{inparaenum}
\item $n/m$ is an integer coprime to $qd$; 
\item $(n+d)/m'$ is an integer coprime to $qd$;
\item $n \equiv r \bmod qd \lcm(m,m')$.
\end{inparaenum}
Clearly the set $P$ is periodic. For $n \in P$ we have
\begin{align*}
	a_n &= a_{n/m} a_{m} = a_{r/n} a_{m},
	&
	a_{n+d} &= a_{(n+d)/m'} a_{m'} = a_{(r + d)/m'} a_{m'},
\end{align*}
and in particular $a_n$ and $a_{n+d}$ are constant on $P$. Since $a_{n+d} = a_{n}$ for almost all $n \in \NN$ and $P$ has positive density or is empty, we conclude that $a_{n+d} = a_{n}$ for all $n \in P$. Taking the union over all possible choices of $m,m',r$ we conclude that  $a_{n+d} = a_{n}$ for all $n \in \NN$.
\end{claimproof} \qedhere

\end{proof}

\bibliographystyle{alphaabbr}
\bibliography{bibliography}

\end{document}